\documentclass[preprint]{imsart}

\startlocaldefs
\usepackage{t1enc}

\usepackage[latin2]{inputenc}
\usepackage{amssymb}
\usepackage{amsmath,amsthm}
\usepackage{color}
\usepackage{mathrsfs}
\usepackage{latexsym}
\usepackage{epsfig}
\usepackage[a4paper,left=4cm,right=3cm,top=3cm,bottom=3cm]{geometry}
\usepackage{subfigure}
\usepackage{hyperref}

\long\def\beginskip#1\endskip{}
\def\endskip{}
\newcommand{\diamm}{\text{diam$_n$}}
\newcommand{\diam}{\mathop{\rm diam}\nolimits}
\newcommand{\ra}{\rightarrow}
\newcommand{\da}{\downarrow}

\renewcommand{\k}{\kappa}

\newcommand{\E}{\mathord{\rm E}}
\newcommand{\Var}{\mathop{\rm var}\nolimits}
\let\var\Var

\newcommand{\eps}{\varepsilon}

\newcommand{\RR}{\mathbb{R}}

\renewcommand{\phi}{\varphi}

\newcommand{\given}{\,|\,}
\renewcommand{\l}{\lambda}

\renewcommand{\th}{\theta}
\renewcommand{\t}{\tau}
\newcommand{\z}{\zeta}
\newcommand{\e}{\varepsilon}

\renewcommand{\d}{\delta}

\renewcommand{\a}{\alpha}
\newcommand{\argmax}{\mathop{\rm argmax}}
\newcommand{\argmin}{\mathop{\rm argmin}}
\newcommand{\loglog}{\mathop{\rm loglog}\nolimits}

\newcommand{\simple}{\widehat\tau_S}
\newcommand{\mmle}{\widehat\tau_M}
\newcommand{\1}{\textbf{1}}

\newtheorem{theorem}{Theorem}[section]

\newtheorem{proposition}[theorem]{Proposition}
\newtheorem{lemma}[theorem]{Lemma}

\theoremstyle{definition}

\newtheorem{ex}[theorem]{Example}

\newtheorem{cond}{Condition}

\numberwithin{equation}{section}

\graphicspath{{./Figures/}}

\endlocaldefs

\begin{document}

\begin{frontmatter}

% "Title of the paper"
\title{Adaptive posterior contraction rates for the horseshoe}
\runtitle{Adaptive contraction for the horseshoe}

% indicate corresponding author with \corref{}
% \author{\fnms{John} \snm{Smith}\corref{}\ead[label=e1]{smith@foo.com}\thanksref{t1}}
% \thankstext{t1}{Thanks to somebody} 
% \address{line 1\\ line 2\\ printead{e1}}
% \affiliation{Some University}

\author{\fnms{St\'ephanie} \snm{van der Pas}\ead[label=e1]{svdpas@math.leidenuniv.nl}\thanksref{t1,m1}},
\author{\fnms{Botond} \snm{Szab\'o}\ead[label=e2]{b.t.szabo@math.leidenuniv.nl}\thanksref{t1,t2,m1,m2}}
\and
\author{\fnms{Aad} \snm{van der Vaart}\ead[label=e3]{avdvaart@math.leidenuniv.nl}\thanksref{t2,m1}}
\address{Leiden University\thanksmark{m1} 
and Budapest University of Technology and Economics\thanksmark{m2}\\
\printead{e1}\\ \printead{e2}\\\printead{e3}}

\thankstext{t1}{Research supported by the Netherlands Organization for Scientific Research.}
\thankstext{t2}{The research leading to these results has received funding from the European
  Research Council under ERC Grant Agreement 320637.}

\runauthor{van der Pas et al.}

\begin{abstract}
We investigate the frequentist properties of  Bayesian procedures for estimation based on the horseshoe prior in the sparse multivariate normal means model. Previous theoretical results assumed that the sparsity level, that is, the number of signals, was known. We drop this assumption and characterize the behavior of the maximum marginal likelihood estimator (MMLE) of a key parameter of the horseshoe prior. We prove that the MMLE is an effective estimator of the sparsity level, in the sense that it leads to (near) minimax optimal estimation of the underlying mean vector generating the data. Besides this empirical Bayes procedure, we consider the hierarchical Bayes method of putting a prior on the unknown sparsity level as well. We show that both Bayesian techniques
lead to rate-adaptive optimal posterior contraction, which implies that the horseshoe posterior is a good candidate for generating rate-adaptive credible sets.
\end{abstract}



\end{frontmatter}

\section{Introduction}
The rise of big datasets with few signals, such as gene expression data and astronomical images, has given an impulse to the study of sparse models. The sequence model, or sparse normal means problem, is  well studied. In this model, a random
vector $Y^n = (Y_1, \ldots, Y_n)$ with values in $\mathbb{R}^n$ is observed, and 
each single observation $Y_i$ is the sum of a fixed mean $\theta_{0,i}$ and standard normal noise $\varepsilon_i$:
\begin{equation}
\label{EqObservation}
Y_i = \th_{0,i} + \varepsilon_i, \quad i = 1, \ldots, n.
\end{equation}
We perform inference on the mean vector $\th_0
=(\th_{0,1},\ldots,\th_{0,n})$, and assume it to be sparse in the nearly black sense, 
meaning that all except an unknown number $p_n = \sum_{i=1}^n\1\{\th_{0,i} \neq 0\}$ of the means are zero. 
We assume that $p_n$ increases with $n$, but not as fast as $n$: $p_n \to \infty$ and $p_n/n \to 0$ as $n$ tends to infinity. 

Many methods to recover $\th_0$ have been suggested. Those most directly related
to this work are \cite{Tibshirani1996, Johnstone2004, Castillo2012,
  Castillo2015, Jiang2009,Griffin2010, Johnson2010, Ghosh2015, Caron2008, Bhattacharya2014, Bhadra2015,
  Rockova2015}. In the present paper we study the Bayesian method based on the \emph{horseshoe prior}
\cite{Carvalho2010, Carvalho2009, Scott2011, Polson2012, Polson2012-2}. Under this prior the
coordinates $\th_1,\ldots, \th_n$ are an i.i.d.\ sample from a scale mixture of normals with a half-Cauchy prior on the
variance, as follows. Given a ``global hyperparameter'' $\t$,
\begin{equation}
\begin{aligned}
\th_i \given \l_i, \t &\sim \mathcal{N}(0,\l_i^2\t^2),\\
\l_i &\sim C^+(0,1), \qquad i = 1, \ldots, n.
\end{aligned}
\label{EqHorseShoePrior}
\end{equation}
In the Bayesian model the observations $Y_i$  follow \eqref{EqObservation} with $\th_0$ taken equal to $\th$.
The \emph{posterior distribution} is then as usual obtained as the conditional distribution of $\th$ given $Y^n$.
For a given value of $\t$, possibly determined by an empirical Bayes method, 
aspects of the posterior distribution of $\th$, such as its mean and variance, can be computed with the help
of analytic formulas and numerical integration \cite{Polson2012, Polson2012-2, vdPas}. 
It is also possible to equip $\t$ with a hyper prior, and follow a hierarchical, full Bayes approach.
Several MCMC samplers and  software packages are available for 
computation of the posterior distribution \cite{Scott2010-2, Makalic2015,  Gramacy2014, horseshoepackage, fasthorseshoe}. %updated software references

The horseshoe posterior has performed well in simulations \cite{Carvalho2010, Carvalho2009, Polson2012,
  Polson2010, Bhattacharya2014, Armagan2013}. 
Theoretical investigation in \cite{vdPas} shows that the parameter $\t$ can, up to a logarithmic factor, 
be interpreted as the fraction of nonzero parameters $\th_i$. 
In particular, if $\t$ is chosen to be at most of the order $(p_n/n)\sqrt{\log{n/p_n}}$, 
then the horseshoe posterior contracts to the true parameter at the (near) minimax rate of recovery for quadratic loss 
over sparse models  \cite{vdPas}. While motivated by these good properties,
we also believe that the results obtained give insight in the performance of Bayesian procedures
for sparsity in general.

In the present paper we make three novel contributions. First and second we establish the contraction rates of the
posterior distributions of $\th$ in the hierarchical, full Bayes case and in the general empirical Bayes case. Third
we study the particular empirical Bayes method of estimating $\t$ by the method of maximum  Bayesian marginal likelihood. 

As the parameter $\t$ can be viewed as measuring sparsity, the first two contributions
are both focused on adaptation to the number $p_n$ of nonzero means, which is unlikely to
be known in practice. The hierarchical and empirical Bayes  methods
studied here are shown to have similar performance, both in theory and in a small simulation study,
and appear to outperform the ad-hoc estimator introduced in \cite{vdPas}. The horseshoe posterior attains similar contraction 
rates as the spike-and-slab priors, as obtained in \cite{Johnstone2004,Castillo2012, Castillo2015}, 
and two-component mixtures, as in \cite{Rockova2015}.
We obtain these results under general conditions on the hyper prior on $\t$, and for general empirical Bayes
methods.

The conditions for the empirical Bayes method are met in particular by the 
maximum marginal likelihood estimator (MMLE). This is the maximum likelihood estimator
of $\t$ under the assumption that the ``prior'' \eqref{EqHorseShoePrior} is part of the data-generating
model, leaving only $\t$ as a parameter. The MMLE is a natural estimator and is easy to compute.
It turns out that the ``MMLE plug-in posterior distribution'' closely mimics the hierarchical Bayes posterior distribution, as has been observed in other settings  \cite{SzVZ,rousseau:szabo:2015}. 
Besides practical benefit, this correspondence provides 
a theoretical tool to analyze the hierarchical Bayes method, which need
not rely on testing arguments (as in \cite{Ghosal2000,GhosalLemberandvanderVaart(2008),vdVvZGamma}).

In the Bayesian framework the spread of the posterior distribution
over the parameter space is used as an indication of the error in estimation.
For instance, a set of prescribed posterior probability around the center of the posterior distribution
(a credible set) is often used in the same way as a confidence region for the parameter. In the follow-up paper \cite{coveragepaper}, we investigate the coverage properties and sizes of the adaptive credible balls and marginal credible intervals.

The paper is organized as follows. We first introduce the MMLE in Section~\ref{sec:MMLE}. Next we
present contraction rates  in Section~\ref{sec:contraction}, for general empirical and hierarchical Bayes approaches, 
and specifically for the MMLE. We illustrate the results in Section~\ref{sec:simulation}. 
We conclude with appendices containing all proofs not given in the main text.

\subsection{Notation} 
We use $\Pi(\cdot\given Y^n,\t)$ for the posterior distribution of $\th$ relative to the
prior \eqref{EqHorseShoePrior} given fixed $\t$, and $\Pi(\cdot\given Y^n)$ for the posterior distribution
in the hierarchical setup where $\t$ has received a prior.
The empirical Bayes ``plug-in posterior'' is the first object with a data-based variable $\widehat\t_n$ substituted for $\t$. 
In order to stress that this does not entail conditioning on $\widehat\t_n$, we also write 
$\Pi_\t(\cdot\given Y^n)$ for $\Pi(\cdot\given Y^n,\t)$, and then $\Pi_{\widehat \t_n}(\cdot\given Y^n)$ is the
empirical Bayes (or plug-in) posterior distribution.

The density of the standard normal distribution is denoted by $\phi$. Furthermore,
$\ell_0[p] = \{\th \in \mathbb{R}^n : \sum_{i=1}^n\1\{\th_{i} \neq 0\} \leq p\}$
denotes the class of nearly black vectors, and we abbreviate
$$\z_\t = \sqrt{2\log(1/\t)}, \qquad \t_n(p) = (p/n)\sqrt{\log (n/p)}, \qquad \t_n= \t_n(p_n).$$

\section{Maximum marginal likelihood estimator} 
\label{sec:MMLE}
In this Section we define the MMLE and compare it to a naive empirical Bayes estimator
previously suggested in \cite{vdPas}. In Section~\ref{sec:EBcontract},
we show that the MMLE is close to the ``optimal'' value  $\t_n(p_n) = (p_n/n)\sqrt{\log(n/p_n)}$ with high
probability, and leads to posterior contraction at the near-minimax rate.

The marginal prior density of a parameter $\th_i$ in the model \eqref{EqHorseShoePrior}
is given by 
\begin{equation}
g_{\t}(\th)=\int_0^{\infty}\phi\left(\frac{\th}{\l\t}\right)\frac{1}{\l\t}\frac{2}{\pi(1+\l^2)}\, d\l.
\label{EqPriorDensityTheta}
\end{equation}
In the Bayesian model the observations $Y_i$ are distributed according to the convolution of this density
and the standard normal density. The MMLE is the maximum likelihood estimator of $\t$ in this latter model, given by
\begin{align}\label{def: MMLE_HS}
\mmle =  \argmax_{\t\in \left[{1}/{n},1\right]} \prod_{i=1}^n\int_{-\infty}^{\infty} \phi(y_i-\th)g_{\t}(\th)\,d\th.
\end{align}
The restriction of the MMLE to the interval $[1/n, 1]$ can be motivated by the
interpretation of $\t$ as the level of sparsity, as in \cite{vdPas}, which makes the interval
correspond to assuming that at least one and at most all  parameters are nonzero. The lower bound of $1/n$ has the additional
advantage of preventing computational issues that arise when $\t$ is very small
(\cite{vdPas,Datta2013}). We found the observation in \cite{Datta2013} that an empirical Bayes approach cannot replace a
hierarchical Bayes one, because the estimate of $\t$ tends to be too small, too general.
In both our theoretical study as in our simulation results the restriction that the MMLE 
be at least $1/n$ prevents a collapse to zero. Our simulations, presented in
Section~\ref{sec:simulation}, also give no reason to believe that the hierarchical Bayes method is inherently
better than empirical Bayes. Indeed, they behave very similarly (depending on the prior on $\t$).

The MMLE requires one-dimensional maximization and is thus easily computed. The behavior of the quantity to be maximized in \eqref{def: MMLE_HS} and the MMLE itself is illustrated in Figure \ref{fig:MMLEComp}. A function for computation is available in the R package 'horseshoe' (\cite{horseshoepackage}). %added this figure

\begin{figure}
\subfigure{{\includegraphics[width=0.45\textwidth]{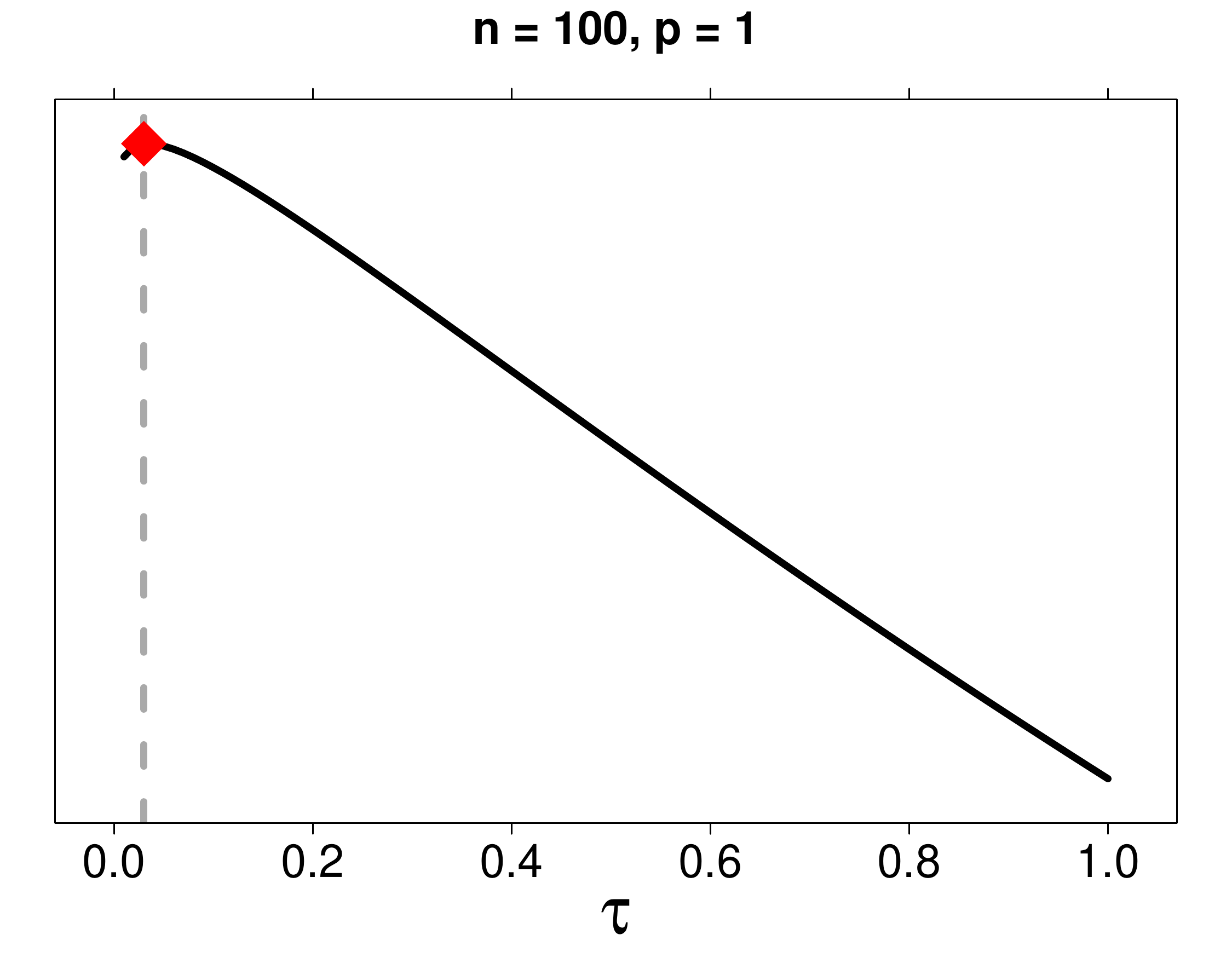} }}%
    \subfigure{{\includegraphics[width=0.45\textwidth]{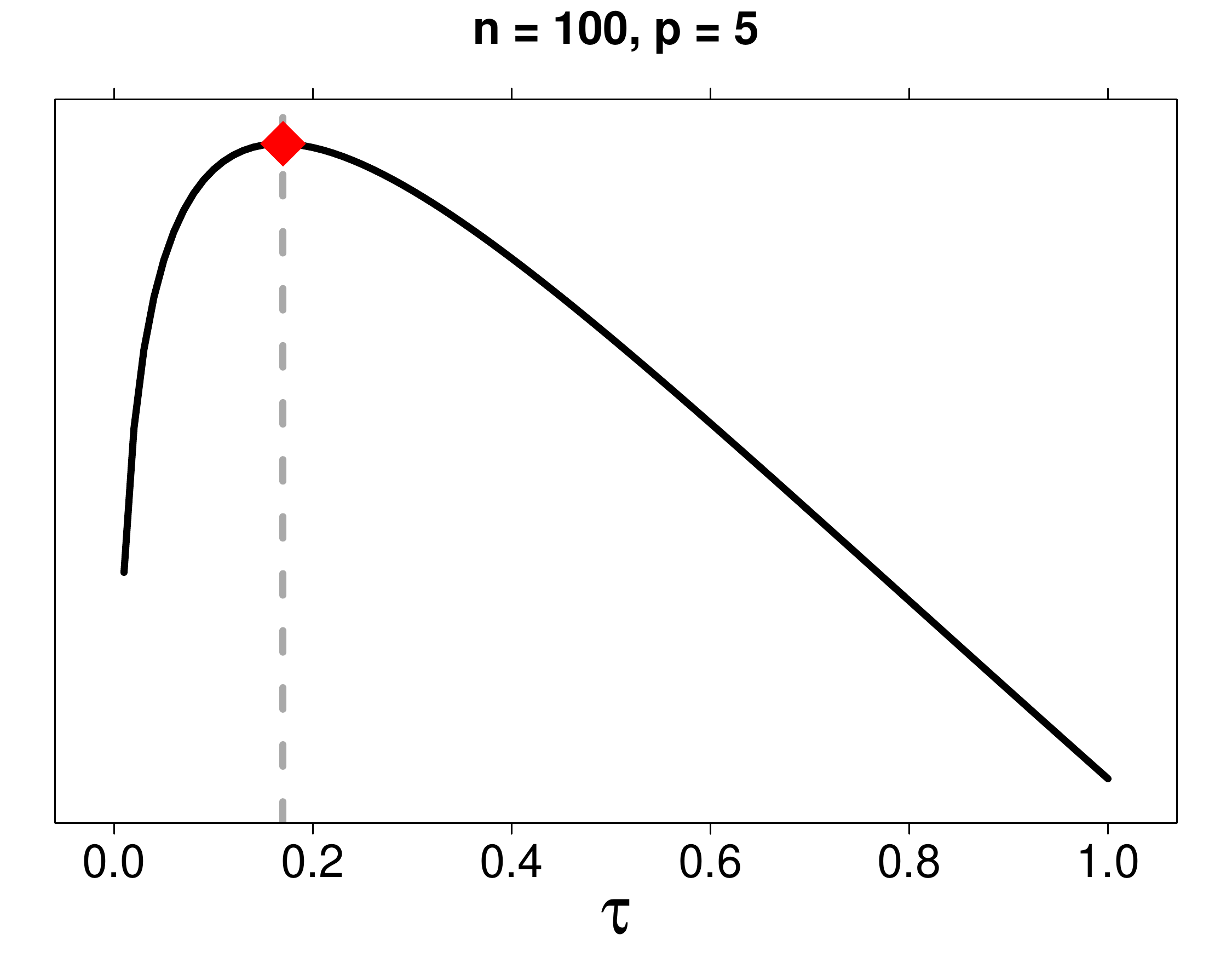} }}\\%
    \subfigure{{\includegraphics[width=0.45\textwidth]{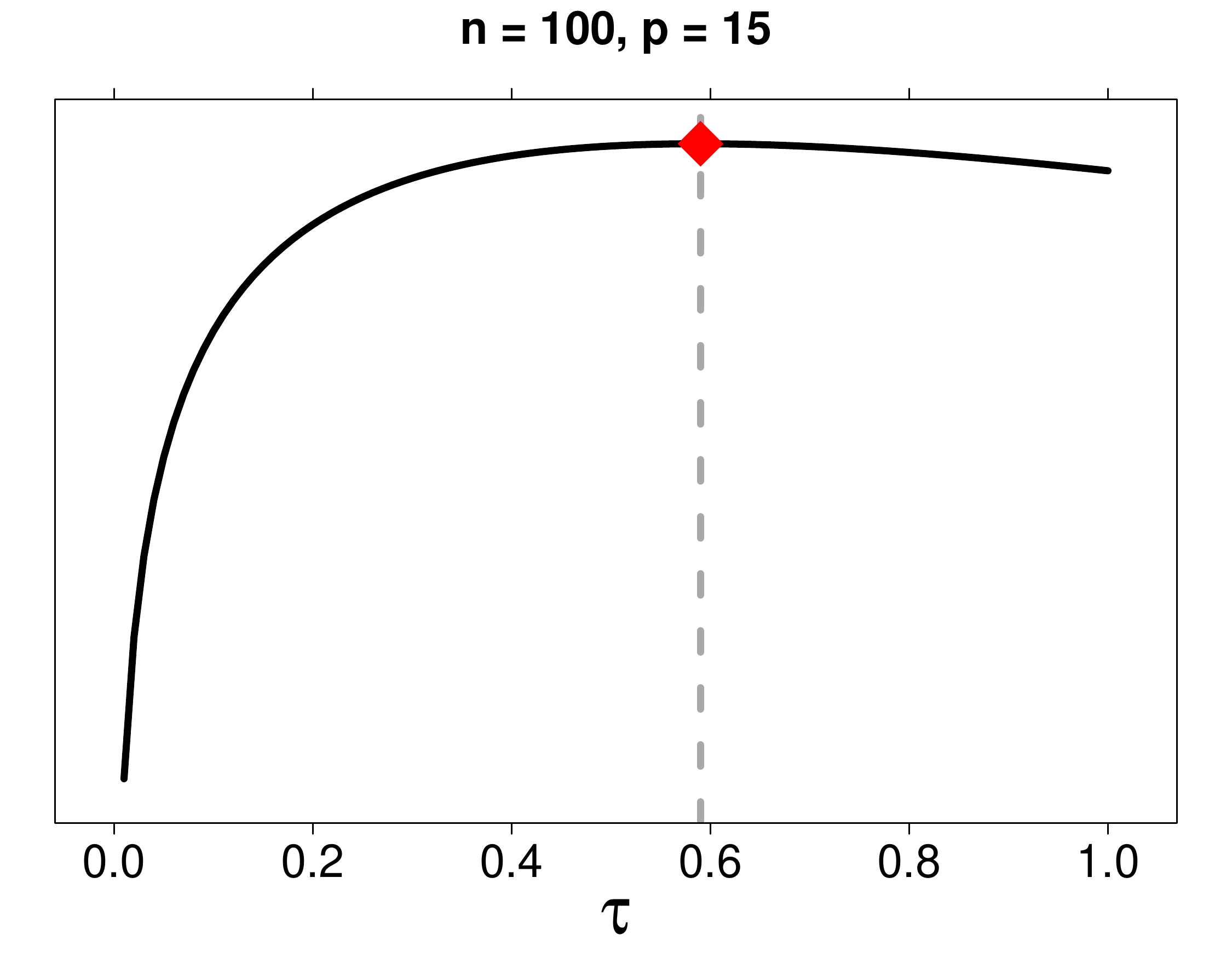} }}%
    \subfigure{{\includegraphics[width=0.45\textwidth]{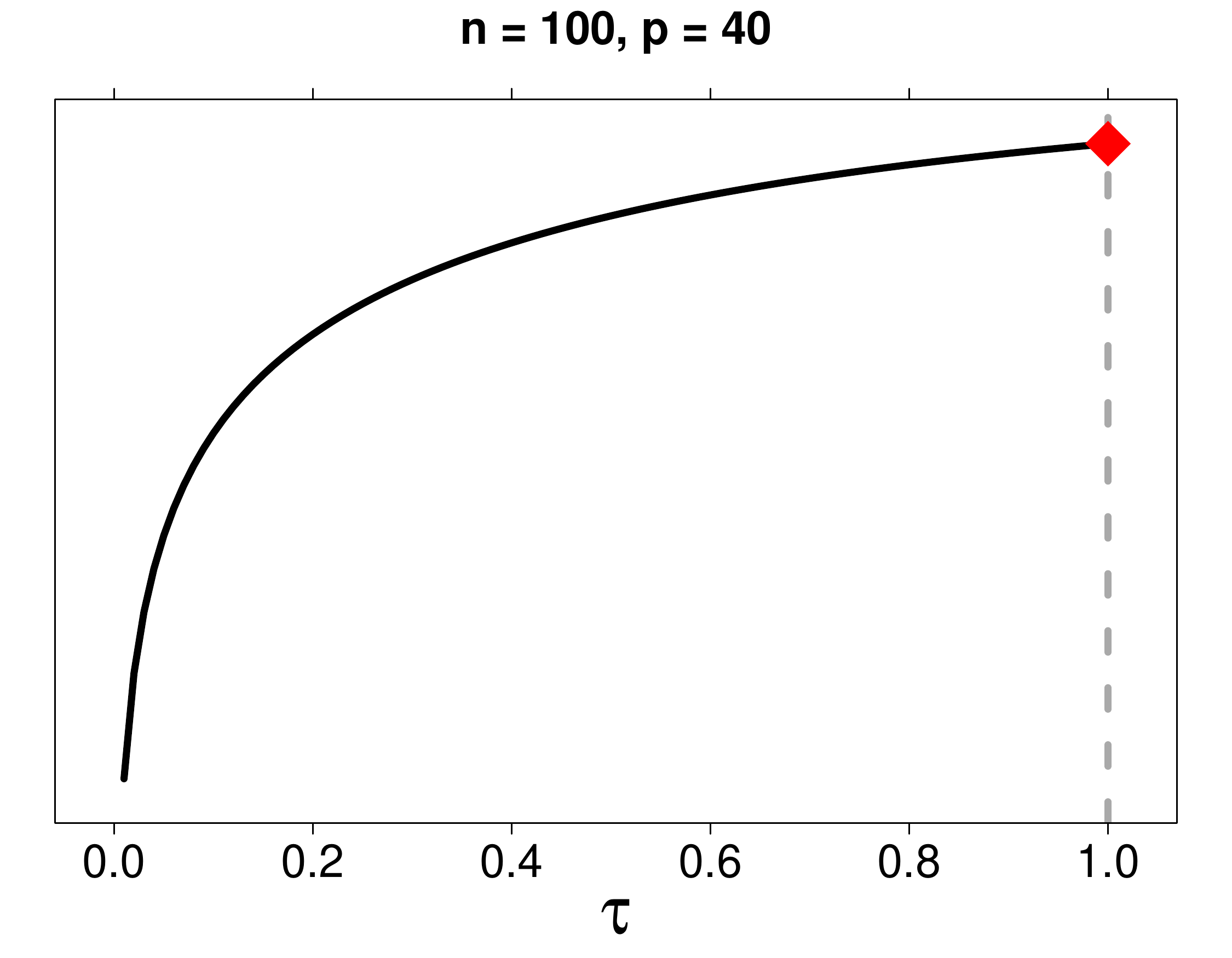} }}%
    \caption{Logarithm of the quantity to be maximized in \eqref{def: MMLE_HS}. The red dot indicates the location of the MMLE. Each plot was made using a single simulated data set consisting of 100 observations each. From left to right, top to bottom, there are 1, 5, 15 or 40 means equal to 10; the remaining means are equal to zero.}
    \label{fig:MMLEComp}
\end{figure}

An interpretation of $\t$ as the fraction of nonzero coordinates motivates
another estimator (\cite{vdPas}), which is based on a count of the number of observations 
that exceed the ``universal threshold'' $\sqrt{2\log n}$:
\begin{equation}\label{eq:def.simple.estimator}
\simple(c_1, c_2) = \max\left\{ \frac{\sum_{i=1}^n \1\{|y_i| \geq \sqrt{c_1\log n} \}}{c_2 n}, \frac{1}{n}\right\},
\end{equation}
where $c_1$ and $c_2$ are positive constants. 
If $c_2 > 1$ and ($c_1 > 2$ or  $c_1 = 2$ and $p_n \gtrsim \log{n}$), then the plug-in posterior distribution with 
the \emph{simple estimator} $\simple(c_1,c_2)$ contracts at the near square minimax rate $p_n\log n$ (see \cite{vdPas}, Section~4).
This also follows from Theorem~\ref{thm:eb_contract} in the present paper, 
as $\simple(c_1,c_2)$ satisfies Condition~\ref{cond.eb} below.

In \cite{vdPas}, it was observed that the simple estimator is prone to underestimation of the sparsity level if signals are smaller than the universal threshold. This is corroborated by the numerical study presented in Figure~\ref{fig:compare.estimators}.
The figure shows approximations to the expected values of $\simple$ and $\mmle$ when $\th_0$ is a vector of length
$n = 100$, with $p_n$ coordinates drawn from a $\mathcal{N}(A, 1)$ distribution, with $A \in \{1, 4, 7\}$,
and the remaining coordinates drawn from a $\mathcal{N}(0, 1/4)$ distribution. For this sample size the
``universal threshold''  $\sqrt{2\log n}$  is approximately 3, and thus signals with $A = 1$ should be
difficult to detect, whereas those with $A = 7$ should be easy;  those with $A = 4$ represent a
boundary case.

\begin{figure}[h]
\includegraphics[width = 0.7\textwidth]{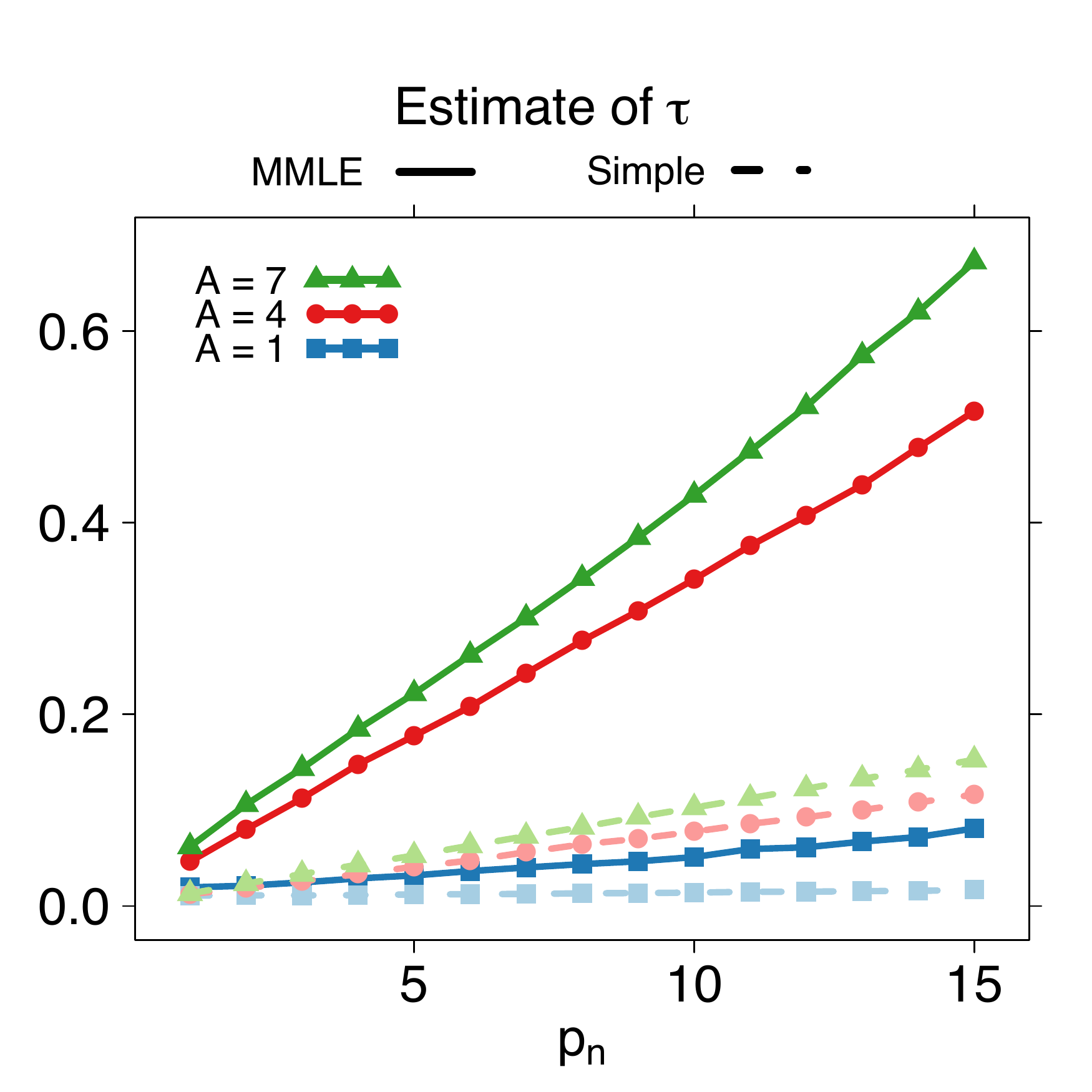}
\caption{Approximate expected values of the MMLE \eqref{def: MMLE_HS} 
(solid) and the simple estimator \eqref{eq:def.simple.estimator}
with $c_1=2$ and $c_2=1$ (dotted) when $p_n$ (horizontal axis) 
out of $n = 100$ parameters are drawn from a $\mathcal{N}(A, 1)$ distribution, 
and the remaining $(n-p_n)$ parameters from a $\mathcal{N}(0, 1/4)$ distribution. 
The study was conducted with $A = 1$ ($\blacksquare$), $A = 4$ ($\bullet$) and $A = 7$ ($\blacktriangle$). 
The results as shown are the averages over $N = 1000$ replications.}
\label{fig:compare.estimators}
\end{figure} 

The figure shows that in all cases the MMLE \eqref{def: MMLE_HS}  yields larger estimates of $\t$ than the simple estimator
\eqref{eq:def.simple.estimator}, and thus leads to less shrinkage. This is expected in light of the results in the following section, 
which show that the MMLE is of order $\t_n(p_n)$, whereas the simple estimator is capped at $p_n/n$. 
Both estimators appear to be linear in the number of nonzero coordinates of $\th_0$, with different slopes. 
When the signals are below the universal threshold, then the simple estimator is unlikely to detect 
any of them, whereas the MMLE may still pick up some of the signals. 
We study the consequences of this for the mean square errors in Section~\ref{sec:simulation}.

\section{Contraction rates}
\label{sec:contraction}
In this section we establish the rate of contraction 
of both the empirical Bayes and full Bayes posterior distributions.   The
\emph{empirical Bayes posterior} is found by replacing $\t$
in the posterior distribution $\Pi(\cdot\given Y^n,\t)$ of $\th$ relative to the prior
\eqref{EqHorseShoePrior} with a given $\t$ by a data-based estimator $\widehat\t_n$; we denote
this by $\Pi_{\widehat\t_n}(\cdot\given Y^n)$. The
\emph{full Bayes posterior} $\Pi(\cdot\given Y^n)$ is the ordinary posterior distribution of $\th$ in the
model where $\t$ is also equipped with a prior and \eqref{EqHorseShoePrior} is interpreted as the
conditional prior of $\th$ given $\t$.

The rate of contraction refers to properties of these posterior distributions when the vector
$Y^n$ follows a normal distribution on $\RR^n$ with mean $\th_0$ and covariance the identity.
We give general conditions on the empirical Bayes estimator $\widehat\t_n$ and the hyper prior on $\t$ that ensure that
the square posterior rate of contraction to $\th_0$ of the resulting posterior distributions
is the near minimax rate $p_n\log n$ for estimation of $\th_0$ relative to the Euclidean norm. 
We also show that these conditions are met by the MMLE and natural hyper priors on $\t$.

The minimax rate, the usual criterion for point estimators, has proven to be a useful benchmark
for the speed of contraction of posterior distributions as well. The posterior cannot contract faster to the truth
than at the minimax rate \cite{Ghosal2000}. The square minimax
$\ell_2$-rate for the sparse normal means problem is $p_n\log(n/p_n)$ \cite{Donoho1992}. This
is slightly faster (i.e.\ smaller) than $p_n \log n$, 
but equivalent if the true parameter vector is not very sparse (if $p_n\le n^\alpha$, for some
$\a<1$, then $(1-\alpha )p_n\log n\le p_n\log (n/p_n)\le p_n\log n$).
For adaptive procedures, where the number of nonzero means $p_n$ is unknown, results are usually given
in terms of the ``near-minimax rate''  $p_n\log n$, for example for the spike-and-slab Lasso
\cite{Rockova2015}, the Lasso \cite{Bickel2009}, and the horseshoe \cite{vdPas}.

\subsection{Empirical Bayes}
\label{sec:EBcontract}
The empirical Bayes posterior distribution achieves the near-minimax contraction rate
provided that the estimator $\widehat\t_n$ of $\t$ satisfies the following condition. 
Let $\t_n(p)=(p/n)\sqrt{\log (n/p)}$.
 
\begin{cond}\label{cond.eb}
There exists a constant $C > 0$ such that $\widehat\t_n\in[1/n, C\t_n(p_n)]$, with $P_{\th_0}$-probability tending to one,
uniformly in $\th_0\in\ell_0[p_n]$. 
\end{cond}

This condition is weaker than the condition given in  \cite{vdPas} for $\ell_2$-adaptation of the empirical
Bayes posterior mean, which requires asymptotic concentration of $\widehat\t_n$ on the same interval
$[1/n, C\t_n(p_n)]$ but at a rate. In
\cite{vdPas} a plug-in value for $\t$ of order $\t_n(p_n)$ was found to be 
the largest value of $\t$ for which the posterior distribution contracts at the minimax-rate, and has 
variance of the same order. Condition \ref{cond.eb} can  be interpreted
as ensuring that $\widehat\t_n$ is of at most this ``optimal'' order. The lower bound can be interpreted
as assuming that there is at least one nonzero mean, which is reasonable in light of the
assumption $p_n \rightarrow \infty$.  In addition, it prevents computational issues, as discussed in 
Section~\ref{sec:MMLE}.

A main result of the present paper is that the MMLE satisfies Condition~\ref{cond.eb}.

\begin{theorem}\label{thm: UB_tau}
The MMLE \eqref{def: MMLE_HS}  satisfies Condition~\ref{cond.eb}.
\end{theorem}

\begin{proof} See Appendix \ref{proof_MMLE_upper}.
\end{proof}

A second main result is that under Condition~\ref{cond.eb} the posterior contracts at the near-minimax rate.

\begin{theorem}\label{thm:eb_contract}
For any estimator $\widehat\t_n$ of $\t$ that satisfies Condition~\ref{cond.eb},
the empirical Bayes posterior distribution contracts around the true parameter at the near-minimax rate:
for any $M_n \ra \infty$ and $p_n\ra\infty$,
\begin{align*}
\sup_{\th_0\in\ell_0[p_n]}
\E_{\th_0}\Pi_{\widehat\t_n}\Big(\th:\|\th_0-\th\|_2\geq M_n \sqrt{p_n \log n} \given Y^n\Big)\rightarrow0.
\end{align*}
In particular, this is true for $\widehat\t_n$ equal to the MMLE.
\end{theorem}

\begin{proof}
See Appendix \ref{proof:eb_contract}.
\end{proof}

\subsection{Hierarchical Bayes}
\label{sec:hierarchical}
The full Bayes posterior distribution contracts at the near minimax rate whenever the prior density
$\pi_n$ on $\t$ satisfies the following two conditions.

%For $\z_\t = \sqrt{2\log(1/\t)}$, define numbers $t_n$ as the solutions to 
%\begin{equation}\label{eq:def.t.n}
%\frac{C_e}{2}\frac{n-p_n}{\z_{t_n}}=t_n^{-1}p_nC_u,
%\end{equation}
%where the constants $C_e$ and $C_u$ are as in Proposition \ref{prop:Em(Y)} and Lemma~\ref{lem: E_nonzero_m(Y)}, respectively.

\begin{cond}\label{cond.hyper.3}
The prior density $\pi_n$ is supported inside $[1/n, 1]$. 
\end{cond}

\begin{cond}\label{cond.hyper.1}
Let $t_n= C_u\pi^{3/2}\, \t_n(p_n)$, with the constant $C_u$ as in 
Lemma~\ref{lem: E_nonzero_m(Y)}(i). The prior density $\pi_n$ satisfies
\begin{equation*}
\int_{t_n/2}^{t_n}\pi_n(\t)\, d\t \gtrsim e^{-cp_n},\quad \text{for some $c\leq C_u/2$}.
\end{equation*}
\end{cond}

The restriction of the prior distribution to the interval $[1/n, 1]$ can be motivated by the same reasons as 
discussed under the definition of the MMLE in Section \ref{sec:MMLE}. 
In our simulations (also see \cite{vdPas}) we have also noted that 
large values produced by for instance a sampler using a half-Cauchy prior, as in the original set-up proposed by \cite{Carvalho2010}, 
were not beneficial to recovery.

As $t_n$ is of the same order as $\t_n(p_n)$, Condition~\ref{cond.hyper.1} is similar to
Condition~\ref{cond.eb} in the empirical Bayes case. It requires that there is sufficient prior mass
around the ``optimal'' values of $\t$.  The condition is satisfied by many
prior densities, including the usual ones, except in the very sparse case that $p_n\lesssim \log n$,
when it requires that $\pi_n$ is unbounded near zero. For this situation we also introduce the
following weaker condition, which is still good enough for a contraction rate with additional
logarithmic factors.

\begin{cond}\label{cond.hyper.1a}
For $t_n$ as in Condition~\ref{cond.hyper.1} the prior density $\pi_n$ satisfies,
\begin{equation*}
\int_{t_n/2}^{t_n}\pi_n(\t)\, d\t \gtrsim t_n.
\end{equation*}
\end{cond}

\begin{ex} 
The Cauchy distribution on the positive reals, truncated to $[1/n, 1]$,
has density $\pi_n(\t) = (\arctan(1) - \arctan(1/n))^{-1}(1+\t^2)^{-1}\textbf{1}_{\t \in [1/n, 1]}$. This satisfies
Condition~\ref{cond.hyper.3}, of course, and Condition~\ref{cond.hyper.1a}. It also satisfies 
the stronger Condition~\ref{cond.hyper.1} provided $t_n\ge e^{-cp_n}$, i.e.\ $p_n\ge C\log n$, for a sufficiently
large $C$. 
\end{ex}

\begin{ex}
For the uniform prior on $[1/n, 1]$, with density $\pi_n(\t) = n/(n-1)\textbf{1}_{\t \in [1/n, 1]}$,
the same conclusions hold.
\end{ex}

\begin{ex}
For the prior with density $\pi_n(x)\propto 1/x$ on $[1/n,1]$, 
Conditions~\ref{cond.hyper.3} and~\ref{cond.hyper.1} hold provided $p_n\gg \loglog n$.
\end{ex}

The following lemma is a crucial ingredient of the derivation of the contraction rate.
It shows that the posterior distribution of $\t$ will concentrate its mass at most a constant multiple of $t_n$
away from zero.  We denote the posterior distribution of $\t$ by the same
general symbol $\Pi(\cdot\given Y^n)$.

\begin{lemma}\label{lem:HBhyper}
If Conditions~\ref{cond.hyper.3} and~\ref{cond.hyper.1} hold, then 
\begin{align*}
\inf_{\th_0\in\ell_0[p_n]}\E_{\th_0}\Pi(\t: \t\le  5t_n \given Y^n)\ra 1.
\end{align*}
Furthermore, if only Conditions~\ref{cond.hyper.3} and~\ref{cond.hyper.1a} hold, then 
 the similar assertion  is true but with $5t_n$ replaced by $(\log n)t_n$.
\end{lemma}

\begin{proof}
See Appendix \ref{proof:lem:HBhyper}.
\end{proof}

We are ready to state the posterior contraction result for the full Bayes posterior. 

\begin{theorem}\label{thm:hierarchical_contract}
If the prior on $\t$ satisfies Conditions~\ref{cond.hyper.3} and~\ref{cond.hyper.1}, then 
the hierarchical Bayes posterior contracts to the true parameter at the near minimax rate:
for any $M_n\ra\infty$ and $p_n\ra\infty$,
\begin{align*}
\sup_{\th_0\in\ell_0[p_n]}\E_{\th_0}\Pi(\th:\|\th-\th_0\|_2\geq M_n\sqrt{p_n\log n}  \given Y^n)\rightarrow0.
\end{align*}
If the prior on $\t$ satisfies only Conditions~\ref{cond.hyper.3} and~\ref{cond.hyper.1a}, then 
this is true with $\sqrt{p_n\log n}$ replaced by $\sqrt{p_n}\log n$. 
\end{theorem}

\begin{proof}
Using the notation $r_n=\sqrt{p_n\log n}$, we can decompose the left side of the preceding display as
\begin{align*}
& \E_{\th_0}\Bigl[\int_{\t\leq 5t_n}+\int_{\t>5t_n}\Bigr] \Pi_\t(\th: \|\th-\th_0\|_2\geq M_nr_n \given Y^n)
\, \pi(\t \given Y^n)\,d\t\\
&\quad\leq  \E_{\th_0}\sup_{\t\leq 5t_n} \Pi_\t(\th:\|\th-\th_0\|_2\geq M_nr_n \given Y^n)
+\E_{\th_0}\Pi(\t: \t> 5 t_n \given Y^n).
\end{align*}
The first term on the right tends to zero by Theorem~\ref{thm:eb_contract}, and the
second by Lemma~\ref{lem:HBhyper}.
\end{proof}

\section{Simulation study}
\label{sec:simulation}
We study the relative performances of the empirical Bayes and hierarchical Bayes approaches further through simulation studies, 
extending the simulation study in \cite{vdPas}.
We  consider the mean square error (MSE) for empirical Bayes combined with either (i) the simple estimator (with $c_1 = 2, c_2 = 1$)
or (ii) the MMLE, and for hierarchical Bayes with either (iii) a Cauchy prior on $\t$, or (iv) a Cauchy prior truncated to $[1/n, 1]$ on $\t$. 

We created a ground truth $\th_0$ of length $n = 400$ with $p_n \in \{20, 200\}$, where each nonzero mean was fixed to
$A \in \{1, 2, \ldots, 10\}$. We computed the posterior mean for each of the four procedures, and
approximated the MSE by averaging over $N = 100$ iterations. The
results are shown in Figure~\ref{fig:MSE}. In addition the figure shows the MSE separately for the
nonzero and zero coordinates of $\th_0$, and the average value (of the posterior mean) of $\t$. 

\begin{figure}[h!]
\begin{center}
\includegraphics[width = 0.9\textwidth]{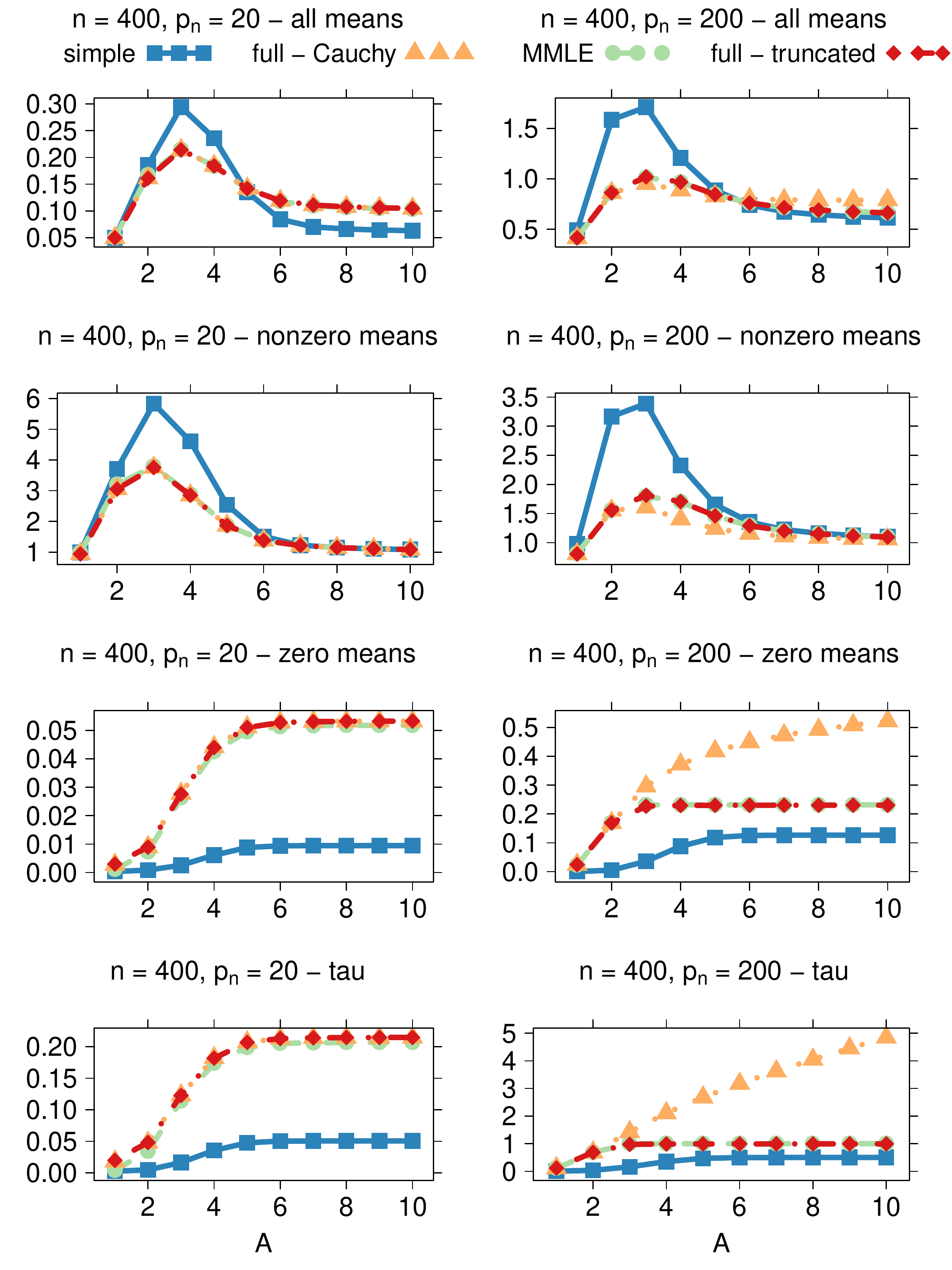}
\caption{Mean square error (overall, for the nonzero coordinates, and for the zero coordinates) of
  the posterior mean corresponding to empirical Bayes with the simple estimator with
  $c_1 = 2, c_2 = 1$ ($\blacksquare$) or the MMLE ($\bullet$) and to hierarchical Bayes with a Cauchy
  prior on $\t$ ($\blacktriangle$) or a Cauchy prior truncated to $[1/n, 1]$ ($\blacklozenge$). The
  bottom plot shows the average estimated value of $\t$ (or the posterior mean in the case of the
  hierarchical Bayes approaches). The settings are $n = 400$ and $p_n = 20$ (left) and $p_n = 200$
  (right); the results are approximations based on averaging over $N = 100$ samples for each value of $A$.}
\label{fig:MSE}	
\end{center}
\end{figure}

The shapes of the curves of the overall MSE for methods (i) and (iii) were discussed in
\cite{vdPas}.  Values close to the threshold $\sqrt{2\log{n}} \approx 3.5$ pose the most difficult
problem, and hierarchical Bayes with a Cauchy prior performs better below the threshold, while
empirical Bayes with the simple estimator performs better above, as the simple estimator is very
close to $p_n/n$ in those settings, whereas the values of $\t$ resulting from hierarchical Bayes are
much larger.

Three new features stand out in this comparison, with the MMLE and hierarchical Bayes with a
truncated Cauchy added in, and the opportunity to study the zero and nonzero means separately. The
first is that empirical Bayes with the MMLE and hierarchical Bayes with the Cauchy prior truncated
to $[1/n, 1]$ behave very similarly, as was expected from our proofs, in which the comparison of the
two methods is fruitfully explored.

Secondly, while in the most sparse setting ($p_n = 20$), full Bayes with the truncated and non-truncated Cauchy priors yield very similar results, as the mean value of $\t$ does not come close to the `maximum' of 1 in either approach, the truncated Cauchy (and the MMLE) offer an improvement over the non-truncated Cauchy in the less sparse ($p_n = 200$) setting. The non-truncated Cauchy does lead to lower MSE on the nonzero means close to the threshold, but overestimates the zero means due to the large values of $\t$. With the MMLE and the truncated Cauchy, the restriction to $[1/n, 1]$ prevents the marginal posterior of $\t$ from concentrating too far away from the 'optimal' values of order $\t_n(p_n)$, leading to better estimation results for the zero means, and only slightly higher MSE for the nonzero means. 

Thirdly, the lower MSE of the simple estimator for large values of $A$ in case $p_n = 20$ is mostly due to a small improvement in estimating the zero means, compared to the truncated Cauchy and the MMLE. As so many of the parameters are zero, this leads to lower overall MSE. However, close to the threshold, the absolute differences between these methods on the nonzero means can be quite large, and the simple estimator performs worse than all three other methods for these values.

Thus, from an estimation point of view, empirical Bayes with the MMLE or hierarchical Bayes with a truncated Cauchy seem to deliver the best results, only to be outperformed by hierarchical Bayes with a non-truncated Cauchy in a non-sparse setting with all zero means very close to the universal threshold.

% AOS,AOAS: If there are supplements please fill:
%\begin{supplement}[id=suppA]
%  \sname{Supplement A}
%  \stitle{Title}
%  \slink[url]{http://lib.stat.cmu.edu/aoas/???/???}
%  \sdescription{Some text}
%\end{supplement}

\appendix

\section{Proof of the main result about the MMLE}

\subsection{Proof of Theorem \ref{thm: UB_tau}}
\label{proof_MMLE_upper}
By its definition the MMLE maximizes the logarithm of the marginal likelihood function, which is given by
\begin{align}\label{eq:def.M.tau}
M_{\t}(Y^n)&=\sum_{i=1}^n\log \Big(\int_{-\infty}^{\infty} \phi(y_i-\th)g_{\t}(\th)d\th\Big).
%\\&=\sum_{i=1}^{n}\log\Big( e^{-y_i^2/2}\int_0^{1}z^{-1/2}\frac{1}{1-(1-1/\t^2)z}e^{y_i^2z/2}d z\Big).
\end{align}
We split the sum in the indices $I_0:=\{i : \th_{0,i}=0\}$ and $I_1:=\{i : \th_{0,i}\not=0\}$. 
By Lemma~\ref{lem: diff_like}, with  $m_\t$ given by \eqref{def: m},
$$\frac{d}{d\t}M_{\t}(Y^n)=\frac1\t\sum_{i\in I_0}m_\t(Y_i)+\frac1\t\sum_{i\in I_1}m_\t(Y_i).$$
By Proposition~\ref{prop:Em(Y)} the expectations of the terms in the first sum are strictly negative and
bounded away from zero for $\t\ge\e$, and any given $\e>0$.
By Lemma~\ref{LemmaNotNormed} the sum behaves likes its expectation, uniformly in $\t$.
By Lemma~\ref{lem: E_nonzero_m(Y)} (i) the function  $m_\t$ is uniformly bounded by a constant $C_u$. 
It follows that for every $\e>0$ there exists a constant $C_\e>0$ such that, for all $\t\ge\e$, 
and with $p_n=\#(\th_{0,i}\not=0)$,
the preceding display is bounded above by 
$$-\frac{n-p_n}{\t}C_\e(1+o_P(1))+\frac{p_n}{\t}C_u.$$
This is negative with probability tending to one as soon as $(n-p_n)/p_n> C_u/C_e$, and in that case the maximum $\mmle$ of $M_\t(Y^n)$ is taken on $[1/n,\e]$. Since this is true  for any $\e>0$, we conclude that 
$\mmle$  tends to zero in probability.

We can now apply Proposition~\ref{prop:Em(Y)} and Lemma~\ref{lem: uniform0} to obtain
the more precise bound on the derivative when $\t\ra 0$ given by
\begin{equation}
\frac{d}{d\t}M_{\t}(Y^n)\le -\frac{(n-p_n)(2/\pi)^{3/2}}{\z_\t}(1+o_P(1))+ \frac{p_n}{\t}C_u.
\label{eq: UB_diff_like}
\end{equation}
This is negative for $\t/\z_\t\gtrsim p_n/(n-p_n)$, and then $\mmle$ is situated on the
left side of the solution to this equation, or $\mmle/\z_{\mmle}\lesssim p_n/(n-p_n)$,
which implies, that $\mmle \lesssim \t_n$, given the assumption that $p_n = o(n)$.

\section{Proofs of the contraction results}

\begin{lemma}
\label{LemmaBoundsPostMeanVariance}
For $A>1$ and every $y\in \RR$,
\begin{itemize}
\item[(i)] $|\E(\th_i\given Y_i=y,\t)-y|\le 2 \z_\t^{-1}$, for $|y|\ge A\z_\t$, as $\t\ra0$.
\item[(ii)] $|\E(\th_i\given Y_i=y,\t)|\le |y|$.
\item[(iii)] $|\E(\th_i\given Y_i=y,\t)|\le \t |y| e^{y^2/2}$, as $\t\ra0$.
\item[(iv)] $|\var(\th_i\given Y_i=y,\t)- 1|\le \z_\t^{-2}$, for $|y|\ge A\z_\t$, as $\t\ra0$.
\item[(v)] $\var(\th_i\given Y_i=y,\t)\le 1+y^2$,
\item[(vi)] $\var (\th_i\given Y_i=y,\t)\lesssim \t e^{y^2/2}(y^{-2}\wedge 1)$, as $\t\ra0$.
\end{itemize}
\end{lemma}

\begin{proof}
Inequalities  (iii) and  (v) come from Lemma~A.2 and Lemma A.4 
in  \cite{vdPas}, while (ii), (iv) and  (vi) are implicit in the proofs of Theorems~3.1 and 3.2 (twice) in \cite{vdPas},
and (i) with the bound $\z_\t$ instead of $\z_\t^{-1}$ is their (17). 
Alternatively, the posterior mean and variance in these assertions are given in \eqref{EqPosteriorCumulants}
 and \eqref{EqPosteriorVarianceExpression}. Then (ii) and (iv) are immediate
from the fact that $0\le I_{3/2}\le I_{1/2}\le I_{-1/2}$, while (iii) and (vi) follow by bounding $I_{-1/2}$ below by a multiple of
$1/\t$ and $I_{3/2}\le I_{1/2}$ above by $(1\wedge y^{-2})e^{y^2/2}$, using Lemmas~\ref{LemmaI-1/2} and~\ref{LemmaI1/23/2}.
Assertions  (i) and (iv) follow from expanding $I_{-1/2}$ and $I_{1/2}$ and $I_{3/2}$, 
again using Lemmas~\ref{LemmaI-1/2} and~\ref{LemmaI1/23/2}. 
\end{proof}

For the proof of Theorem \ref{thm:eb_contract}, we use the following observations.  The posterior density of $\th_i$ given $(Y_i=y,\t)$ is (for fixed $\t$) an exponential family with density
$$\th\mapsto \frac{\phi(y-\th)g_\t(\th)}{\psi_\t(y)}=c_\t(y) e^{\th y} g_\t(\th) e^{-\th^2/2},$$
where $g_\t$ is the posterior density of $\th$ given in \eqref{EqPriorDensityTheta}, and $\psi_\t$ is the
Bayesian marginal density of $Y_i$, given in \eqref{Eqpsitau}, and the norming constant is given by
$$c_\t(y)=\frac{\phi(y)}{\psi_\t(y)}=\frac{\pi}{\t I_{-1/2}(y)},$$
for the function $I_{-1/2}(y)$ defined in \eqref{eq:def.Ik}. The cumulant moment generating function
$z\mapsto \log \E (e^{z\th_i}\given Y_i=y,\t)$ of the family is given by 
$z\mapsto\log \bigl(c_\t(y)/c_\t(y+z)\bigr)$, which is $z\mapsto \log I_{-1/2}(y+z)$ plus an additive constant
independent of $z$. We conclude that the first, second and fourth cumulants are given by
\begin{align}
\hat\th_i(\t)=\E(\th_i\given Y_i=y,\t)&=\frac{d}{dy}\log I_{-1/2}(y),\nonumber\\
\var(\th_i\given Y_i=y,\t)&=\frac{d^2}{dy^2}\log I_{-1/2}(y),
\label{EqPosteriorCumulants}\\
\E\bigl[\bigl(\th_i-\hat\th_i(\t)\bigr)^4\given Y_i=y,\t\bigr]-3 \var(\th_i\given Y_i=y,\t)^2
&=\frac{d^4}{dy^4}\log I_{-1/2}(y).\nonumber
\end{align}
The derivatives at the right side can be computed by repeatedly using the product and sum rule together with the
identity $I_k'(y)=y I_{k+1}(y)$, for $I_k$ as in \eqref{eq:def.Ik}. In addition, since $(\log h)''=h''/h-(h'/h)^2$, for any function $h$,  and $I_{-1/2}'(y)=y I_{1/2}(y)$ and $I_{-1/2}''(y)=y^2I_{3/2}(y)+I_{1/2}(y)$, we
have
\begin{equation}
\label{EqPosteriorVarianceExpression}
\var(\th_i\given Y_i=y,\t)=y^2\Bigl[\frac{I_{3/2}}{I_{-1/2}}-\Bigl(\frac{I_{1/2}}{I_{-1/2}}\Bigr)^2\Bigr](y)+\frac{I_{1/2}}{I_{-1/2}}(y).
\end{equation}

\subsection{\label{proof:eb_contract}Proof of Theorem \ref{thm:eb_contract}}
\begin{proof}
Set $r_n=\sqrt{p_n\log n}$ and $\t_n=\t_n(p_n)$. By  Condition~\ref{cond.eb} and the triangle inequality,
\begin{align*}
&\E_{\th_0}\Pi_{\widehat\t_n}\Big(\th:\|\th_0-\th\|_2\geq M_n r_n \given Y^n\Big)\\
&\qquad\le \E_{\th_0}\1_{\widehat\t_n\in[1/n, C\t_n] }\Pi_{\widehat\t_n}\Big(\th:\|\th_0-\hat\th(\widehat\t_n)\|_2+\|\th-\hat\th(\widehat\t_n)\|_2\geq M_n r_n \given Y^n\Big)+o(1)\nonumber\\
&\qquad \leq \E_{\th_0} \sup_{\t\in[1/n, C\t_n] }\Pi_{\t}\Big(\th:\|\th_0-\hat\th(\t)\|_2+\|\th-\hat\th(\t)\|_2\geq M_n r_n \given Y^n \Big)+o(1).
\end{align*}
Hence, in view of Chebyshev's inequality, it is sufficient to show that, 
with $\var(\th\given Y^n,\t)=\E\bigl(\|\th-\hat\th(\t)\|^2\given Y^n,\t\bigr)$,
\begin{align}
P_{\th_0} \Bigl(\sup_{\t\in [1/n, C\t_n]}\|\th_0-\hat\th(\t)\|_2\geq (M_n/2)r_n\Bigr)&=o(1),\label{eq: MSE}\\
P_{\th_0} \Bigl(\sup_{\t\in [1/n, C\t_n]} \var (\th \given Y^n, \t )\geq M_n r_n^2\Bigr)&=o(1).\label{eq: spread}
\end{align}
To prove \eqref{eq: MSE} we first use Lemma~\ref{LemmaBoundsPostMeanVariance}(i)+(ii) to see that
$|\hat\th_i(\t)|\lesssim \z_\t$ and next the triangle inequality to see that $|\hat\th_i(\t)-\th_{0,i}|\lesssim \z_\t+ |Y_i-\th_{0,i}|$, as $\t\ra0$.
This shows that 
\begin{align}
\E_{\th_{0,i}}\sup_{\t\in [1/n,\t_n]}(\th_{0,i}-\hat\th_{i}(\t))^2\lesssim \sup_{\t\geq1/n}\z_\t^2+\var_{\th_{0,i}}Y_i\lesssim \log n.\label{eq: help101}
\end{align}
Second we use Lemma~\ref{LemmaBoundsPostMeanVariance} (iii) and (ii) to see that $|\hat\th_i(\t)|$ is bounded
above by $\t|Y_i|e^{Y_i^2/2}$ if $|Y_i|\le \z_{\t_n}$ and bounded above by $|Y_i|$ otherwise, so that
$$\E_0\sup_{\t\in [1/n,C\t_n]} |\hat\th_i(\t)|^2\lesssim \int_0^{\z_{\t_n}}(C\t_n)^2y^2e^{y^2}\phi(y)\,dy+\int_{\z_{\t_n}}^\infty y^2\phi(y)\,dy
\lesssim \t_n\z_{\t_n}.$$
Applying the upper bound \eqref{eq: help101} for the $p_n$ non-zero coordinates $\th_{0,i}$, and the upper bound in the
last display for the zero parameters, we find that
\begin{align*}
\E_{\th_{0}}\sup_{\t \in [1/n, C\t_n]}\|\th_{0}-\hat\th(\t)\|_2^2\lesssim p_n \log{n}+(n-p_n)\t_n\z_{\t_n}\lesssim p_n \log{n}.
\end{align*}
Next an application of Markov's inequality leads to \eqref{eq: MSE}.

The proof of \eqref{eq: spread} is similar. For the nonzero $\th_{0,i}$ we use
the fact that $\Var(\th_i \given Y_i,\t)\leq 1+\z_{\t}^2 \lesssim \log n$,
by Lemma~\ref{LemmaBoundsPostMeanVariance} (iv) and (v), 
while for the zero $\th_{0,i}$ we use that $\Var(\th_i \given Y_i,\t)$ is bounded above by $\t e^{Y_i^2/2}$ for $|Y_i|\le\z_{\t_n}$ and
bounded above by $1+Y_i^2$ otherwise, by Lemma~\ref{LemmaBoundsPostMeanVariance} (vi) and (v).
For the two cases of parameter values this gives bounds for $\E_{\th_{0,i}}\sup_{\t\in[1/n,C\t_n]}\Var(\th_i \given Y_i,\t)$ of the same
form as the bounds for the square bias, resulting in the overall bound 
$p_n \log{n}+(n-p_n)\t_n\z_{\t_n}\lesssim p_n \log{n}$ for the sum of these variances.
An application of Markov's inequality gives \eqref{eq: spread}.
\end{proof}

\subsection{\label{proof:lem:HBhyper}Proof of Lemma \ref{lem:HBhyper}}
The number $t_n$ defined in Condition~\ref{cond.hyper.1a} is the (approximate) solution to the equation
$p_nC_u/\t=C_e(n-p)/(2\z_\t)$, for $C_e=(\pi/2)^{3/2}$. By the decomposition \eqref{eq: UB_diff_like},  
with $P_{\th_0}$-probability tending to one, 
$$\frac{\partial}{\partial\t} M_{\t}(Y^n)<
\begin{cases}
p_nC_u/(t_n/2),& \text{ if } t_n/2\leq\t\leq t_n,\\
0& \text{ if }\t>t_n,\\
-p_nC_u/(2t_n),& \text{ if }\t\geq 2t_n.
\end{cases}$$
Therefore, for $M_\t(Y^n)$ defined in \eqref{eq:def.M.tau}, $\t_{\min}=\argmin_{\t\in[t_n/2,t_n]}M_{\t}(Y^n)$, and  $\t\geq 2t_n$,
\begin{align*}
M_{\t}(Y^n)-M_{\t_{\min}}(Y^n) &=\Bigl[\int_{\t_{\min}}^{t_n}+\int_{t_n}^{2t_n}+\int_{2t_n}^\t\Bigr]\frac{\partial}{\partial s}  M_{s}(Y^n)\,ds\\
&\leq (t_n/2)p_nC_u/(t_n/2)+0 -(\t-2t_n)p_nC_u/(2t_n)\\
&=-(\t-4t_n)p_nC_u/(2t_n) \leq -\t p_n C_u/(10t_n),
\end{align*}
for $\t\ge 5t_n$. Since $\pi(\t\given Y^n)\propto \pi(\t)e^{M_\t(Y^n)}$ by Bayes's formula, with $P_{\th_0}$-probability tending to one, for $c_n\ge 5$
\begin{align*}
\Pi(\t\geq c_nt_n \given Y^n)
&%=\frac{\int_{\t\geq c_nt_n}\pi(\t)e^{M_{\t}(Y^n)}\, d\t}{\int_{\t>1/n}\pi(\t)e^{M_{\t}(Y^n)}\, d\t} 
\leq \frac{\int_{\t\geq c_nt_n}e^{M_{\t_{\min}}(Y^n)-\t p_n C_u/(10t_n)}\pi(\t)\, d\t}{\int_{\t\in [t_n/2,t_n] }e^{M_{\t_{\min}}(Y^n)}\pi(\t)\, d\t}
%&\leq \frac{e^{M_{\t_{\min}}(Y^n)}e^{-\frac{C_et_n}{20}\frac{n-p_n}{\z_{t_n}}}\int_{\t\geq 5t_n}\pi(\t) d\t}{e^{M_{\t_{\min}} (Y^n)}\int_{\t\in [t_n/2,t_n] }\pi(\t) d\t}
\lesssim \frac{e^{-c_np_nC_u/10}}{\int_{\t\in [t_n/2,t_n] }\pi(\t)\, d\t}.%\label{eq: HB_post_UB} 
\end{align*}
Under Condition~\ref{cond.hyper.1} this tends to zero if $c_n\ge 5$. 
Under the weaker Condition~\ref{cond.hyper.1a} this is certainly true for $c_n\ge\log n$.

\section{Lemmas supporting the MMLE results}
\label{Sec: lem: diff_like}

For $k\in\{-1/2,1/2,3/2\}$ define a function $I_k: \RR\to\RR$ by
\begin{align}\label{eq:def.Ik}
I_k(y): =\int_0^1 z^k \frac{1}{\t^2+(1-\t^2)z}e^{y^2z/2}\,dz.
\end{align}
The Bayesian marginal density of $Y_i$ given $\t$ is the convolution $\psi_\t:=\phi\ast g_\t$ of the
standard normal density and the prior density of $g_\t$, given in  \eqref{EqPriorDensityTheta}.
The latter is a half-Cauchy mixture of  normal densities $\phi_{\t\l}$ with mean zero and standard deviation $\t\l$.
By Fubini's theorem it follows that $\psi_\t$ is a half-Cauchy mixture of the densities $\phi\ast\phi_{\t\l}$. In other words
\begin{align}
\psi_\t(y)&=\int _0^\infty \frac{e^{-\frac12y^2/(1+\t^2\l^2)}}{\sqrt{1+\t^2\l^2}\sqrt{2\pi}}\,\frac{2}{1+\l^2}\frac1\pi\,d\l%\nonumber\\
=\int _0^1 \frac{e^{-\frac12y^2(1-z)}}{\sqrt{2\pi}\pi} \frac{\t z^{-1/2}}{\t^2(1-z)+z}\,dz\nonumber\\
&=\frac{\t}{\pi}I_{-1/2}(y)\phi(y),\label{Eqpsitau}
\end{align}
where the second step follows by the substitution $1-z=(1+\t^2\l^2)^{-1}$ and some algebra.
Note that $I_{-1/2}$ depends on $\t$, but this has been suppressed  from the notation $I_k$.

Set
\begin{align}
m_\t(y)=y^2\frac{I_{1/2}(y)-I_{3/2}(y)}{I_{-1/2}(y)}-\frac{I_{1/2}(y)}{I_{-1/2}(y)}.\label{def: m}
\end{align}

\begin{lemma}\label{lem: diff_like}
The derivative of the log-likelihood function takes the form 
\begin{align*}
\frac{d}{d\t}M_\t(y^n)=\frac1\t\sum_{j=1}^{n}m_\t(y_j).
\end{align*}
\end{lemma}

\begin{proof}
From \eqref{Eqpsitau} we infer that, with a dot denoting the partial derivative with respect to $\t$,
\begin{align*}
\frac{\dot\psi_\t}{\psi_\t}&=\frac1\t+\frac{\dot I_{-1/2}}{ I_{-1/2}}=\frac{I_{-1/2}+\t\dot I_{-1/2} }{ \t I_{-1/2}}
=\frac{\int_0^1\frac{e^{y^2z/2}}{\sqrt z N(z)^2}[N(z)-2\t^2(1-z)]\,dz}{ \t I_{-1/2}},
\end{align*}
where $N(z)=\t^2(1-z)+z=\t^2+(1-\t^2)z$. By  integration by parts,
$$y^2(I_{1/2}-I_{3/2})(y)=\int_0^1\frac{\sqrt{z}(1-z)}{N(z)} y^2 e^{y^2z/2}\,dz=-2\int_0^1 e^{y^2z/2}\, d\Bigl[\frac{\sqrt z(1-z)}{N(z)}\Bigr].$$
Substituting the right hand side in formula \eqref{def: m}, we readily see by some algebra that 
$\t^{-1}$ times the latter formula reduces to the right side of the preceding display.
\end{proof}

\begin{proposition}\label{prop:Em(Y)}
Let $Y\sim N(\th,1)$. Then  $\sup_{\t\in [\eps,1]}\E_0 m_\t(Y)<0$ for every $\eps>0$, and as $\t\ra0$,
\begin{equation}\label{eq:m_tau_asymp}
\E_\th m_\t(Y)=
\begin{cases}
-\frac{2^{3/2}}{\pi^{3/2}}\,\frac{\t}{\z_\t}\bigl(1+o(1)\bigr), & |\th|=o(\z_{\t}^{-2}),\\
o (\t^{1/16}\z_{\t}^{-1}), &|\th|\leq \z_\t/4.
\end{cases}
\end{equation}
\end{proposition}

\begin{proof}
Let $\k_\t$ be the solution to the equation ${e^{y^2/2}}/({y^2/2})=1/\t$, that is
$$e^{\k_\t^2/2}=\frac 1\t\, \k_\t^2/2,\qquad\qquad 
\k_\t\sim \z_\t+\frac{2\log\z_\t}{\z_\t},\qquad\qquad \z_\t=\sqrt{2\log (1/\t)}.$$
We split the integral over $(0,\infty)$ into the three parts $(0,\z_\t)$, $(\z_\t,\k_\t)$, and $(\k_\t,\infty)$,
where we shall see that the last two parts give negligible contributions.

By Lemma~\ref{Lem: asymp_m}(vi) and (vii), if $|\th|\k_\t=O(1)$,
\begin{align*}
\int_{|y|\geq \k_\t} m_\t(y)\phi(y-\th)\,dy&\lesssim \int_{z\geq \k_\t-|\th|} \phi(z)\,dz\lesssim \frac{e^{-(\k_\t-\th)^2/2}}{\k_\t-\th}
\lesssim \frac{e^{-\k_\t^2/2}}{\k_\t},\\
\int_{\z_\t\leq |y|\leq\k_\t}m_\t(y)\phi(y-\th)\,dy&\lesssim \int_{\z_\t\leq |y|\leq\k_\t}\frac{\t e^{y^2/2-(y-\th)^2/2}}{y^2}\,dy
\lesssim\frac{\t(\k_\t-\z_\t)}{\z_\t^2}.
\end{align*}
By the definition of $\k_\t$, both terms are of smaller order than $\t/\z_\t$.

Because $e^{y^2/2}/y^2$ is increasing for large $y$ and reaches the value $\t^{-1}/\z_\t^2$ at $y=\z_\t$, Lemma~\ref{LemmaI-1/2}
gives that $I_{-1/2}(y)=\pi\t^{-1}(1+O(1/\z_\t^2))$ uniformly in $y$ in the interval $(0,\z_\t)$. Therefore
\begin{align*}
\int_{|y| \leq \z_\t} m_\t(y)\phi(y-\th)\,dy&=
\int_{0}^{\z_\t}\frac{y^2I_{1/2}(y)-y^2I_{3/2}(y)-I_{1/2}(y)}{\t^{-1}\pi}\,\phi(y)\,dy+R_\t,
\end{align*}
where the remainder $R_\t$ is bounded in absolute value by 
%$\int_0^{\z_\t} |y^2(I_{1/2}-I_{3/2})(y)-I_{1/2}(y)|\Bigl|\frac{\phi(y-\th)}{I_{-1/2}(y)}-\frac{\phi(y)}{\pi/\t}\Bigr|\,dy
$\int_0^{\z_\t} |y^2(I_{1/2}-I_{3/2})(y)-I_{1/2}(y)|\phi(y)\,dy$ times 
$\sup_{0\le y\le\z_\t}\bigl|\phi(y-\th)/(I_{-1/2}(y)\phi(y))-{1}/{(\t^{-1}\pi)}\bigr|$, which is bounded
above by $\t\bigl(\z_\t^{-2}+e^{|\th|\z_\t-\th^2/2}-1)=o(\t\z_\t^{-1})$, for $|\th|=o(\z_\t^{-2})$.
By Lemma~\ref{LemmaI1/23/2} the integrand in the integral 
is bounded above by a constant for $y$ near $0$ and by a multiple
of $y^{-2}$ otherwise, and hence the integral remains bounded. Thus the remainder $R_\t$ is negligible. 
By Fubini's theorem the integral in the preceding display can be rewritten 
\begin{align*}
&\frac{\t}{\pi}\int_0^1 \frac{\sqrt z}{\t^2+(1-\t^2)z}\int_0^{\z_\t}\bigl[y^2(1-z)-1\bigr]\frac{e^{-y^2(1-z)/2}}{\sqrt{2\pi}}\,dy\,dz\\
&\qquad=-\frac{\t}{\pi}\int_0^1 \frac{\sqrt z}{\t^2+(1-\t^2)z}\int_{\z_\t}^\infty\bigl[y^2(1-z)-1\bigr]\frac{e^{-y^2(1-z)/2}}{\sqrt{2\pi}}\,dy\,dz
\end{align*}
by the fact that the inner integral vanishes when computed over the interval $(0,\infty)$ rather than $(0,\z_\t)$. 
Since $\int_y^\infty [(va)^2-1]\phi(va)\,dv=y\phi(ya)$, it follows that the right side is equal to
\begin{align*}
&\qquad  -\frac{\t}{\pi}\int_0^1 \frac{\sqrt z}{\t^2+(1-\t^2)z}\frac{\z_\t \, e^{-\z_\t^2(1-z)/2}}{\sqrt{2\pi}}\,dz.
\end{align*}
We split the integral in the ranges $(0,1/2)$ and $(1/2,1)$. For $z$ in the first range we have $1-z\ge 1/2$,
whence the contribution of this range is bounded in absolute value  by
$$\frac{{\z_\t}\t}{\pi\sqrt{2\pi}} e^{-\z_\t^2/4}\int_0^{1/2} \frac{\sqrt z}{(1-\t^2)z}\,dz=O({\z_\t}\t e^{-\z_\t^2/4}).$$
Uniformly in $z$ in the range $(1/2,1)$ we have $\t^2+(1-\t^2)z \sim z$, and the corresponding contribution is
\begin{align*}
-\frac{\t}{\pi}\int_{1/2}^1 \frac{1}{\sqrt z}\frac{\z_\t \, e^{-\z_\t^2(1-z)/2}}{\sqrt{2\pi}}\,dz
&= -\frac{\t}{ \pi \z_\t \sqrt{2\pi}}\int_0^{\z_\t^2/2} \frac1{\sqrt {1-u/\z_\t^2}}e^{-u/2}\,du.
\end{align*}
by the substitution $\z_\t^2(1-z)=u$. The integral tends to $\int_0^\infty e^{-u/2}\,du=2$,
and hence the expression is asymptotic to half the expression as claimed.

The second statement follows by the same estimates, where now we use that
$e^{|\th|2\z_\t-\th^2/2}\le \t^{-15/16}$, if $|\th|\le \z_\t/4$. 

Since $\E_0m_\t(Y)\sim -c\t/\z_\t$ for a positive constant $c$,  as $\t\da0$, 
the continuous function $\t\mapsto \E_0m_\t(Y)$ is certainly negative
if $\t>0$ and $\t$ is close to zero. To see that it is bounded away from zero as $\t$ moves away from 0,
we computed $\E_0 m_\t(Y)$ via numerical integration. The result is shown in Figure~\ref{fig:expectation}. 
%Assuming that the numerical error of the R \texttt{integrate()} routine is as advertised, the Figure together with the first statement of the proposition indicates that $\E_0 m_\t(Y)$ remains below zero for all $\t \in [0,1]$. 
\begin{figure}[h!]
\begin{center}
\includegraphics[width = 0.9\textwidth]{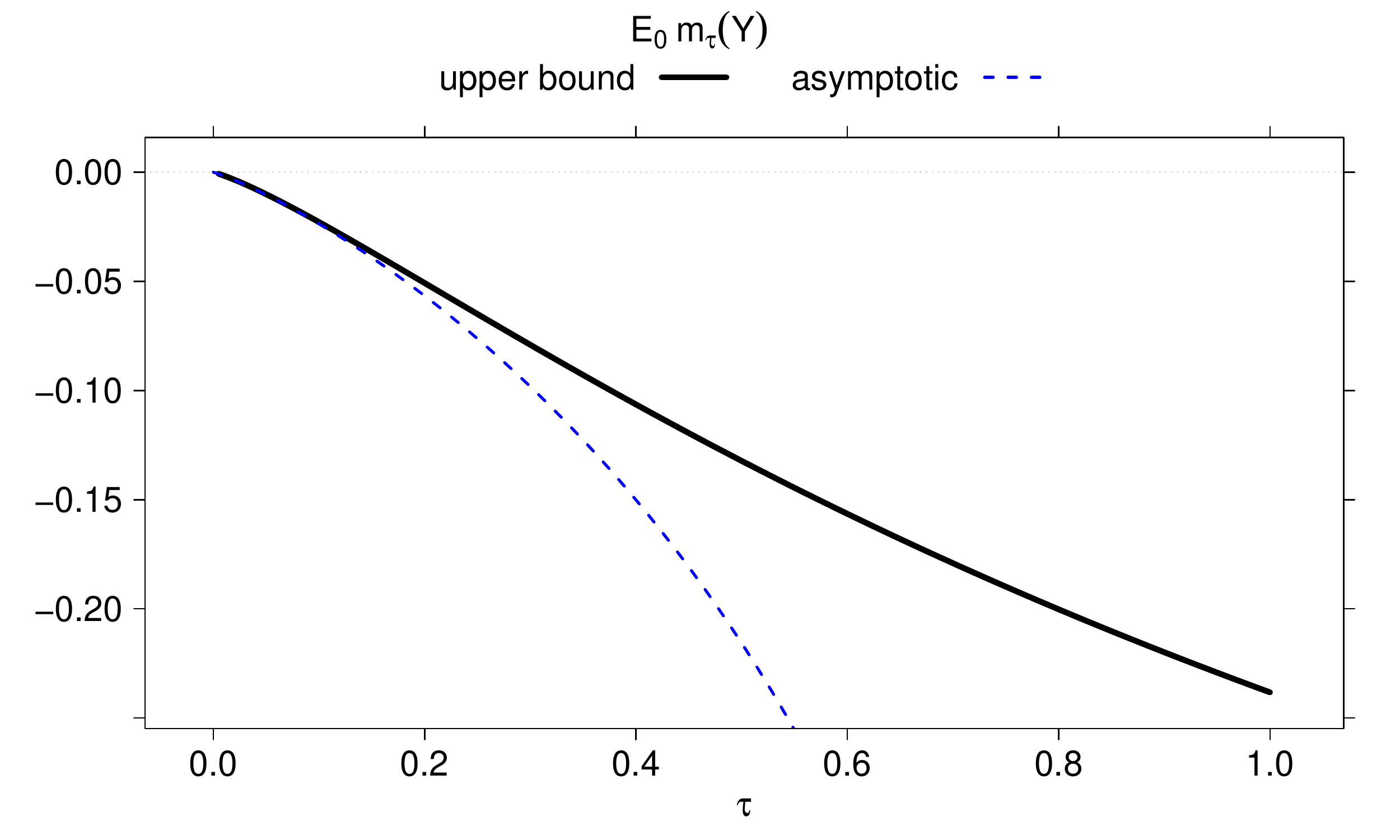} 
\caption{Upper bound on $\E_0 m_\t(Y)$ as computed with the R \texttt{integrate()} routine (solid line). The upper bound $m_\t(y) \leq y^2$ was used for $|y| > 500$ for numerical stability. The dashed line shows the asymptotic value \eqref{eq:m_tau_asymp}.  }
\label{fig:expectation}
\end{center}
\end{figure}
\end{proof}

\begin{lemma}\label{lem: uniform0}
For any $\e_\t\downarrow0$ and uniformly in $I_0\subseteq\{i: |\th_{0,i}|\leq \z_\t^{-1}\}$ with $|I_0|\gtrsim n$,
\begin{align*}
\sup_{1/n\leq \t\leq \e_\t}\frac{1}{|I_0|}\Big|\sum_{i\in I_0}m_\t(Y_i)\frac{\z_\t}{\t}
-\sum_{i\in I_0}\E_{\th_0} m_\t(Y_i)\frac{\z_\t}{\t}\Big|\stackrel{P_{\th_0}}{\rightarrow}0.
\end{align*}
Similarly, uniformly in $I_1\subseteq\{i: |\th_{0,i}|\leq \z_\t/4\}$,
\begin{align*}
\sup_{1/n\leq \t\leq \e_\t}\frac{1}{|I_1|}\Big|\sum_{i\in I_1}m_\t(Y_i)\frac{\z_\t}{\t^{1/32}}
-\sum_{i\in I_1}\E_{\th_0} m_\t(Y_i)\frac{\z_\t}{\t^{1/32}}\Big|\stackrel{P_{\th_0}}{\rightarrow}0.
\end{align*}
\end{lemma}

\begin{proof}
Write $G_n(\t)=|I_0|^{-1}\sum_{i\in I_0}m_\t(Y_i)(\z_\t) /\t$. In view of Corollary 2.2.5 of \cite{vdVW}
(applied with $\psi(x)=x^2$) it is sufficient to show that $\Var_{\th_0} G_n(\t)\rightarrow 0$ for some
$\t$, and
\begin{align}\label{eq: entropy}
\int_0^{\diamm}\sqrt{N(\eps, [1/n,1],d_n)}\,d\eps=o(1),
\end{align}
where $d_n$ is the intrinsic metric defined by its square
$d_n^2(\t_1,\t_2)=\Var_{\th_0}\bigl(G_n(\t_1)-G_n(\t_2)\bigr)$, 
$\diamm$ is the diameter of the interval $[1/n,1]$ with respect to the metric $d_n$, and
$N(\eps,A,d_n)$ is the covering number of the set $A$ with $\eps$ radius balls with respect to the
metric $d_n$. 

If $|\th_{0,i}|\leq \z_\t^{-1}$, then in view of Lemma~\ref{lem: m2}, as $\t\ra0$,
\begin{align*}
\Var_{\th_0} G_n(\t)&\leq \frac1{|I_0|}\E_{\th_0}\bigl( m_\t(Y)\z_\t /\t\bigr)^2=o(\t^{-1}/|I_0|). 
\end{align*}
This tends to zero, as  $\t n\ge 1$ by assumption. 
Combining this with the triangle inequality we also see that the diameter $\diamm$ tends to 0.

Next we deal with the entropy. 
The metric $d_n$ is up to a constant equal to the square root of the left side of \eqref{def: metric}.
By Lemma~\ref{lem: metric} it satisfies
$$d_n(\t_1,\t_2)\lesssim  |I_0|^{-1/2} |\t_2/\t_1-1|\t_1^{-1/2}.$$
To compute the covering number of the interval
$[1/n,1]$, we cover this by dyadic blocks $[2^i/n,2^{i+1}/n]$, for $i=0,1,2,...,\log_2 n$. On the
$i$th block the distance $d_n(\t_1,\t_2)$ is bounded above by a multiple of $n|\t_1-\t_2|/2^{3i/2}$.
We conclude that the $i$th block can be covered by %$n^{-1/2}2^{i/2}\eps^{-1/2}o(1)$ balls. 
a multiple of $\eps^{-1} 2^{-i/2}$ balls of radius $\eps$. Therefore the whole
interval $[1/n,1]$ can be  covered by a multiple of $\eps^{-1}\sum_i 2^{-i/2}\lesssim \eps^{-1}$ balls of radius
$\eps$.
%$$\sum_{i=1}^{\log_2 n}n^{-1/2}2^{i/2}\eps^{-1/2}o(1)= o(1)\eps^{-1}$$
Hence the integral of the entropy is bounded by
$$\int_0^{\diamm}\sqrt{N(\eps, [1/n,1],d_n)}\,d\eps\lesssim \int_0^{\diamm}\eps^{-1/2}\,d\eps. $$
This tends to zero as $\diam_n$ tends to zero.

The second assertion of the lemma follows similarly, where we use the second parts
of Lemmas~\ref{lem: m2} and~\ref{lem: metric}. 
\end{proof}

\begin{lemma}\label{lem: metric}
Let $Y\sim N(\th,1)$. For $|\th|\lesssim \z_\t^{-1}$ and $0<\t_1<\t_2\le 1/2$,
\begin{align}
\label{def: metric}
\E_\th\left(\frac{\z_{\t_1}}{\t_1}m_{\t_1}(Y)-\frac{\z_{\t_2}}{\t_2}m_{\t_2}(Y)\right)^2
&\lesssim (\t_2-\t_1)^2 \t_1^{-3}.
\end{align}
Furthermore, for $|\th|\leq \z_\t/4$, and $\eps=1/16$ and $0<\t_1<\t_2\le 1/2$,
\begin{align*}
\E_\th\left(\frac{\z_{\t_1}}{\t_1^{\eps}}m_{\t_1}(Y)-\frac{\z_{\t_2}}{\t_2^{\eps}}m_{\t_2}(Y)\right)^2
\lesssim (\t_2	-\t_1)^2 \t_1^{-2-\eps}.
\end{align*}
\end{lemma}

\begin{proof}
In view of Lemma~\ref{Lem: Diff} the left side of \eqref{def: metric} is bounded above by,
for $\dot m_\t$ denoting the partial derivative of $m_\t$ with respect to $\t$,
\begin{align*}
&(\t_1-\t_2)^2 
 \sup_{\t\in[\t_1,\t_2]}\E_\th\Big(\frac{\z_{\t}}{\t}\dot m_{\t}(Y)-\frac{\z_{\t}+\z_{\t}^{-1}}{\t^2}m_{\t}(Y)\Big)^2\\
&\qquad\leq (\t_1-\t_2)^2 \Big[2 \sup_{\t\in[\t_1,\t_2]}\E_\th\Big(\frac{\z_{\t}}{\t}\dot m_{\t}(Y)\Big)^2
+ 2 \sup_{\t\in[\t_1,\t_2]}\E_\th\Big(\frac{\z_{\t}+\z_{\t}^{-1}}{\t^2}m_{\t}(Y)\Big)^2\Big].
\end{align*}
By Lemma~\ref{lem: m2} the second expected value on the right hand side 
 is bounded from above by a multiple of $\sup_{\t\in[\t_1,\t_2]} \t^{-3}\lesssim \t_1^{-3}$. 

To handle the first expected value, we note that the partial derivative of
$I_k$ with respect to $\t$ is given by $\dot I_k=2\t( J_{k+1}-J_{k})$, for
\begin{align}
J_k(y)=\int_0^1 \frac{z^k}{(\t^2+(1-\t^2)z)^2}e^{y^2z/2}dz.\label{def: J}
\end{align}
Therefore, by \eqref{def: m},
\begin{align*}
\dot m_{\t}(y)&=(y^2-1) \frac{\dot I_{1/2}}{I_{-1/2}}(y)-y^2 \frac{\dot I_{3/2}}{I_{-1/2}}(y)
-\frac{\dot I_{1/2}}{I_{-1/2}}(y)m_\t(y)\\
&= 2\t\Big[(y^2-1)\frac{J_{3/2}-J_{1/2}}{I_{-1/2}}(y)-y^2 \frac{J_{5/2}-J_{3/2}}{I_{-1/2}}(y)
-\frac{J_{1/2}-J_{-1/2}}{I_{-1/2}}(y)m_{\t}(y)\Big].
\end{align*}
Since $J_{k}\leq I_{k-1}/(1-\t^2)$ and $J_k\le I_k/\t^2$,
and $k\mapsto I_k$ and $k\mapsto J_k$ are decreasing and nonnegative, we have that
\begin{align}
0&\leq\frac{J_{3/2}-J_{5/2}}{I_{-1/2}}\le \frac{J_{1/2}-J_{3/2}}{I_{-1/2}}\le \frac{J_{1/2}}{I_{-1/2}}\leq 4,\nonumber\\
0&\leq\frac{J_{-1/2}-J_{1/2}}{I_{-1/2}}\leq \frac{J_{-1/2}}{I_{-1/2}}\leq \frac1{\t^{2}}.\label{eq: help_metric}
\end{align}
By combining the preceding two displays we conclude
\begin{align}
\E_\th\dot m_{\t}^2(Y)&\lesssim \t^2\Big[1+\E_\th Y^4 + \frac{1}{\t^4}\E_\th m_{\t}^2(Y) \Big].
\end{align}
Here $\E_\th Y^4$ is bounded and $\E_\th m_{\t}^2(Y)$ is bounded above by $\t\z_\t^{-2}$ by 
Lemma~\ref{lem: m2}. It follows that $(\z_\t/\t)^2\E_\th\dot m_{\t}^2(Y)$ is bounded by a multiple
of $\t^{-3}\le \t_1^{-3}$.

For the proof of the second assertion of the lemma, when $|\th|\leq \z_\t/4$,
we argue similarly, but now must bound, 
\begin{align*}
(\t_1-\t_2)^2 \Big[2 \sup_{\t\in[\t_1,\t_2]}\E_\th\Big(\frac{\z_{\t}}{\t^\eps}\dot m_{\t}(Y)\Big)^2
+ 2 \sup_{\t\in[\t_1,\t_2]}\E_\th\Big(\frac{\eps\z_{\t}+\z_{\t}^{-1}}{\t^{1+\eps}}m_{\t}(Y)\Big)^2\Big].
\end{align*}
The same arguments as before apply, now using the second bound from Lemma~\ref{lem: m2}.
\end{proof}

\begin{lemma}\label{lem: m2}
Let $Y\sim N(\th,1)$. Then,  as $\t\rightarrow 0$,
$$\E_\th m_\t^2(Y)=
\begin{cases}
o(\t \z_\t^{-2}),&|\th|\lesssim \z_\t^{-1},\\
o (\t^{1/16} \z_{\t}^{-2}),&|\th|\leq \z_\t/4.
\end{cases}$$
\end{lemma}

\begin{proof}
By Lemma~\ref{Lem: asymp_m} (i), (vi) and (vii) we have, if $|\th|\z_\t\lesssim 1$,
\begin{align*}
\int_{|y|\geq \k_\t} m_\t^2(y)\phi(y-\th)\,dy&\lesssim \int_{|z|\geq\k_\t-\th}^\infty \phi(z)\,dz\lesssim e^{-(\k_\t-\th)^2/2}(\k_\t-\th)^{-1}
\lesssim \t\z_\t^{-3},\\
\int_{\z_\t\leq |y|\leq \k_\t} m_\t^2(y)\phi(y-\th)\,dy&\lesssim \int_{\z_\t}^{\k_\t}\t y^{-2}e^{y^2/2-(y-\th)^2/2}\,dy=\t(\k_\t-\z_\t)\z_\t^{-2},\\
\int_{|y| \leq \z_\t}  m_\t^2(y)\phi(y-\th)\,dy&\lesssim 
\t^2 \int_0^{\z_\t}(y^{-4}\wedge 1)e^{y^2/2}e^{\th\z_\t-\th^2/2}\,dy\lesssim \t \z_\t^{-4}.
\end{align*}
All three expressions on the right are $o(\t\z_\t^{-2})$.

The second assertion of the lemma follows by the same inequalities, together with 
the inequalities $e^{-(\k_\t-\th)^2/2}\le \t^{-9/32}$ and $e^{|\th|2\z_\t-\th^2/2}\le \t^{-15/16}$, if $|\th|\le \z_\t/4$. 
\end{proof}

\begin{lemma}\label{LemmaNotNormed}
If the cardinality of $I_0:=\{i: \th_{0,i}=0\}$ tends to infinity, then
\begin{align*}
\sup_{1/n\leq \t\leq 1}\frac{1}{|I_0|}\Big|\sum_{i\in I_0}m_\t(Y_i)
-\sum_{i\in I_0}\E_{\th_0} m_\t(Y_i)\Big|\stackrel{P_{\th_0}}{\rightarrow}0.
\end{align*}
\end{lemma}

\begin{proof}
By Lemma~\ref{Lemma: tech1}(i) we have that $\E_0m_\t^2(Y_i)\lesssim 1$ uniformly in $\t$ and
by the proof of Lemma~\ref{lem: metric} $\E_0(m_{\t_1}-m_{\t_2})^2(Y_i)\lesssim |\t_1-\t_2|^2/\t_1$,
uniformly in $0<\t_1<\t_2\le 1$. The first shows that the marginal variances of the
process $G_n(\t):=|I_0|^{-1}\sum_{i\in I_0}m_\t(Y_i)$ tend to zero as $|I_0|\ra \infty$. The second
allows to control the entropy integral of the process and complete the proof, in the same way 
as the proof of Lemma~\ref{lem: uniform0}.
\end{proof}

\begin{lemma}
\label{lem: E_nonzero_m(Y)}
\label{Lemma: tech1}
\label{Lem: asymp_m}
The function $y\mapsto m_\t(y)$ is symmetric about 0 and nondecreasing on $[0,\infty)$ with 
\begin{itemize}
\item[(i)] $-1\le  m_\t(y)\le C_u$, for all $y\in\RR$ and all $\t\in[0,1]$, and some $C_u<\infty$.
\item[(ii)] $m_\t(0)=-(2\t/\pi)(1+o(1))$, as $\t\ra0$.
\item[(iii)] $m_\t(\z_\t)=2/(\pi\z_\t^2)(1+o(1))$, as $\t\ra0$.
\item[(iv)] $m_\t(\k_\t)=1/(\pi+1)/(1+o(1))$, as $\t\ra0$.
\item[(v)] $\sup_{y\ge A\z_\t}|m_\t(y)-1|=O(\z_\t^{-2})$, as $\t\ra0$, for every $A>1$.
\item[(vi)] $m_\t(y)\sim\t e^{y^2/2}/(\pi y^2/2+\t e^{y^2/2})$, as $\t\ra0$, uniformly in $|y|\ge1/\e_\t$, for any $\e_\t\da0$.
\item[(vii)] $|m_\t(y)|\lesssim\t e^{y^2/2}(y^{-2}\wedge 1)$, as $\t\ra0$, for every $y$.
\end{itemize}
\end{lemma}

\begin{proof}
As seen in the proof of Lemma~\ref{lem: diff_like} the function $m_\t$ can be written
$$m_\t(y)=1+\t\frac{\dot I_{-1/2}}{I_{-1/2}}(y)=1+2\t^2\int_0^1 \frac{z-1}{\t^2+(1-\t^2)z}g_y(z)\,dz,$$
for $z\mapsto g_y(z)$ the probability density function on $[0,1]$ with $g_y(z)\propto e^{y^2/2}z^{-1/2}/(\t^2+(1-\t^2)z)$.
If $y$ increases, then the probability distribution increases stochastically, and hence so does the expectation of
the increasing function $z\mapsto (z-1)/(\t^2+(1-\t^2)z)$. (More precisely, note that $g_{y_2}/g_{y_1}$ is increasing if $y_2>y_1$ and
apply Lemma~\ref{Lem: DominatingDensity}.)

(i). The inequality $m_\t(y)\ge -1$ is immediate from the definition of \eqref{def: m} of $m_\t$
and the fact that $I_{3/2}\le I_{1/2}\le I_{-1/2}$. For the upper bound it suffices to show
that both $\sup_ym_\t(y)$ remains bounded as $\t\ra 0$ and that $\sup_y\sup_{\t\ge\d} m_\t(y)<\infty$
for every $\d>0$. 

The first follows from the monotonicity and  (v). 

For the proof of the second we note that if $\t\ge\d>0$, then
$\d^2\le \t^2+(1-\t^2)z\le 1$, for every $z\in [0,1]$, so that  the denominators in the
integrands of $I_{-1/2}, I_{1/2}, I_{3/2}$ are uniformly bounded away from zero and infinity and hence
$$m_\t(y) \le y^2\frac{I_{1/2}(y)-I_{3/2}(y)}{I_{-1/2}(y)}
\le\frac1{\d^2}\frac{y^2\int_0^1 \sqrt z(1-z)e^{y^2z/2}\,dz}{\int_0^1  z^{-1/2}e^{y^2z/2}\,dz}.$$
After changing variables $zy^2/2=v$, the numerator and denominator take the forms of
the integrals in the second and first assertions of Lemma~\ref{LemmaIncompleteGamma}, except that
the range of integration is $(0,y^2/2)$ rather than $(1,y)$. In view of the lemma the 
quotient approaches 1 as $y\ra\infty$.  For $y$ in a bounded interval  the leading factor $y^2$ is bounded,
while the integral in the numerator is smaller than the integral in the denominator, as $z(1-z)\le z\le z^{-1/2}$,
for $z\in [0,1]$.

Assertions (ii)-(v) are consequences of the representation \eqref{def: m}, 
Lemmas~\ref{LemmaI-1/2} and~\ref{LemmaI1/23/2} and
the fact that $I_{1/2}(0)=\int_0^1z^{-1/2}dz \big(1+O(\t^2)\big)\ra 2$.

Assertions (vi) and (vii) are immediate from Lemmas~\ref{LemmaI-1/2} and~\ref{LemmaI1/23/2}.
\end{proof}

\subsection{Technical lemmas}

\begin{lemma}
\label{LemmaIncompleteGamma}
For any $k$, as $y\ra\infty$,
$$\int_1^y u^ke^u\,du=y^ke^y\bigl(1-k/y+O(1/y^2)\bigr).$$
Consequently, as $y\ra\infty$,
$$\int_1^y{u^{k}}{e^u}\,du-\frac 1y\int_1^y u^{k+1} e^u\,du=y^{k-1}e^y\bigl(1+O(1/y)\bigr).$$
\end{lemma}

\begin{proof}
By integrating by parts twice, the first integral is seen to be equal to 
$$y^ke^y-e-ky^{k-1}e^y+ke +R,$$
where $R$ satisfies
\begin{align*}
|R|&=|k(k-1)|\int_1^y u^{k-2}e^u\,du\\
&\le |k(k-1)|\int_1^{y/2}(1\vee (y/2)^{k-2})e^u\,du+|k(k-1)|\int_{y/2}^y ((y/2)^{k-2}\vee y^{k-2})e^u\,du\\
&\lesssim |k(k-1)|\Bigl[(1\vee y^{k-2})e^{y/2}+y^{k-2}e^y\Bigr].
\end{align*}
The second assertion follows by applying the first one twice.
\end{proof}

\begin{lemma}
\label{LemmaI-1/2}
There exist functions $R_\t$ with $\sup_{y}|R_\t(y)|=O(\sqrt{\t})$ as $\t\downarrow0$,  such that
\begin{align*}
I_{-1/2}(y)&=\Bigl(\frac{\pi}\t+\sqrt{y^2/2}\int_1^{y^2/2} \frac1{v^{3/2}}e^v\,dv\Bigr)\bigl(1+R_\t(y)\bigr).
\end{align*}
Furthermore, given $\e_\t\ra 0$ there exist functions $S_\t$ with $\sup_{y\ge 1/\e_\t}|S_\t(y)|=O(\sqrt \t+\e_\t^2)$, 
such that, as $\t\da0$, 
\begin{align*}
I_{-1/2}(y)&=\Bigl(\frac{\pi}\t+\frac{e^{y^2/2}}{y^2/2}\Bigr)\bigl(1+S_\t(y)\bigr).
\end{align*}
\end{lemma}

\begin{proof}
For the proof of the first assertion we separately consider the ranges $|y|\le 2\z_\t$ and $|y|>2\z_\t$.
For $|y|\le 2\z_\t$ we split the integral in the definition of $I_{-1/2}$ over the 
intervals $(0,\t)$, $(\t,(2/y^2)\wedge 1)$ and $((2/y^2)\wedge 1,1)$,
where we consider the third interval  empty if $y^2/2\le 1$. Making the changes of coordinates
$z=u\t^2$ in the first integral, and $(y^2/2)z=v$ in the second and third integrals, we see that
\begin{align*}
I_{-1/2}(y)&=\frac1\t\int_0^{1/\t}\frac1{\sqrt u}\frac 1{1+(1-\t^2)u} e^{y^2\t^2u/2}\,du\\
&\qquad+\sqrt{y^2/2}\Bigl[\int_{y^2\t/2}^{y^2/2\wedge 1}+ \int_{y^2/2\wedge 1}^{y^2/2}\Bigr]
\frac1{\sqrt v}\frac 1{\t^2y^2/2+(1-\t^2)v}e^v\,dv
\end{align*}
For $|y|\le 2\z_\t$,  the exponential in the first integral tends to 1, uniformly in $u\le 1/\t$. Since 
$e^u-1\le ue^u$, for $u\ge 0$, replacing it by 1 gives an error of at most
$$\frac 1\t \int_0^{1/\t}\frac1 {\sqrt u}\frac{e^{y^2\t{/2}}y^2\t^2u}{1+(1-\t^2)u}\,du\lesssim \frac1\t y^2\t^{3/2}.$$
As $(1-\t^2)(1+u)\le 1+(1-\t^2)u\le 1+u$, dropping the factor $1-\t^2$ from the denominator
makes a multiplicative error of order $1+O(\t^2)$. 
Since $\int_0^\infty u^{-1/2}/(1+u)\,du=\pi$ and $\int_{1/\t}^\infty u^{-1/2}/(1+u)\,du\lesssim \t^{1/2}$,
the first term gives a contribution of $\pi/\t+O(\t^{-1/2})$, uniformly in $|y|\le2 \z_\t$.
In the second integral we bound the factor $\t^2y^2/2+(1-\t^2)v$ below by $(1-\t^2)v$, the exponential $e^v$ above by 
$e$ and the upper limit of the integral by 1, and next evaluate the integral to be bounded by a constant times $\t^{-1/2}$.
For the third integral we separately consider the cases that $y^2/2\leq1$ and $y^2/2>1$.
In the first case the third integral contributes nothing; the second term (the integral) in the
assertion of the lemma is bounded and hence also contributes a negligible amount relative to $\pi/\t$.
Finally consider the case that $y^2/2>1$. If in  the third integral
we replace $\t^2y^2/2+(1-\t^2)v$ by $v$, we obtain the second term in the assertion of the lemma.
The difference is bounded above by
$$\sqrt{y^2/2} \int_{ 1}^{y^2/2}
\frac1{\sqrt v}\frac {\t^2v+\t^2y^2}{v(\t^2y^2/2+(1-\t^2)v)}e^v\,dv
\lesssim \t^2\sqrt{y^2/2}\int_1^{y^2/2} (v^{-3/2}+y^2 v^{-5/2})e^v\,dv.$$
This is negligible relative to the integral in the assertion.
This concludes the proof of the first assertion of the lemma for the range $|y|\le 2\z_\t$.

For $|y|$ in the interval $(2\z_\t,\infty)$ we split the integral in the definition of
$I_{-1/2}$ into the ranges $[0,1/3]$ and $(1/3,1]$. The contribution of the first range is bounded above by
$$\frac1{\t^2}e^{y^2/6}\int_0^{1/3} z^{-1/2}\,dz\ll {\sqrt\t}\frac{e^{y^2/2}}{y^2/2},$$
for $|y|\ge 2\z_\t$. This is negligible relative to the integral in the assertion, which expands as
$e^{y^2/2}/\sqrt{y^2/2}$, as claimed by the second assertion of the lemma.
In the contribution of the second range we use that $z\le \t^2+(1-\t^2)z\le (1+2\t^2)z$,
for $z\ge 1/3$, and see that this is up to a multiplicative term of order $1+O(\t^2)$ equal to
$$\int_{1/3}^1 z^{-3/2} e^{y^2z/2}\,dz=\sqrt{y^2/2}\Bigl[\int_1^{y^2/2}-\int_1^{y^2/6}\bigr] v^{-3/2}e^v\,dv.$$
Applying Lemma~\ref{LemmaIncompleteGamma}, we see that the contribution of the second integral 
is bounded above by a multiple of $(y^2/2)^{-1}e^{y^2/6}$, which is negligible relative to the first. 

To prove the second assertion of the lemma we expand the integral in the first assertion 
with the help of Lemma~\ref{LemmaIncompleteGamma}. 
\end{proof}

\begin{lemma}
\label{LemmaI1/23/2}
For $k>0$, there exist functions $R_{\t,k}$ with $\sup_{y}|R_{\t,k}(y)|=O(\t^{2k/(k+1)})$,
and for given $\e_\t\ra 0$ functions $S_{\t,k}$ with $\sup_{y\ge 1/\e_\t}|S_{\t,k}(y)|=O(\t^{2k/(2k+1)}+\e_\t^2)$, such that,
as $\t\da0$,
\begin{align*}
I_{k}(y)&=\frac1{(y^2/2)^k}\int_0^{y^2/2}v^{k-1}e^v\,dv\bigl(1+R_{\t,k}(y)\bigr)\lesssim\bigl(1\wedge y^{-2}\bigr)e^{y^2/2},\\
I_{k}(y)&=\frac{e^{y^2/2}}{y^2/2}\bigl(1+S_{\t,k}(y)\bigr).
\end{align*}
There also exist functions $\bar R_{\t}$ with $\sup_{y}|\bar R_{\t}(y)|=O(\t^{1/2})$ and  
$\bar S_{\t}$ with $\sup_{y\ge 1/\e_\t}|\bar S_{\t}(y)|=O(\sqrt\t+\e_\t^2)$, such that, as $\t\da0$ and $\e_\t\ra0$,
\begin{align*}
I_{1/2}(y)-I_{3/2}(y)&=\frac1{\sqrt{y^2/2}}\int_0^{y^2/2}\frac{1-2v/y^2}{\sqrt v}e^v\,dv\bigl(1+\bar R_{\t}(y)\bigr)
\lesssim (1\wedge y^{-4})e^{y^2/2},\\
I_{1/2}(y)-I_{3/2}(y)&=\frac{e^{y^2/2}}{(y^2/2)^2}\bigl(1+\bar S_{\t}(y)\bigr).
\end{align*}
\end{lemma}

\begin{proof}
%We give the proof only for the function $I_{1/2}-I_{3/2}$, the proof for $I_{k}$ being similar, but easier. 
We split the integral in the definition of $I_k$ over the intervals $[0,\t^a]$ and $[\t^a,1]$, for $a=2/(k+1)$.
The contribution of the first integral is bounded above by
$$e^{\t^a y^2/2}\int_0^{\t^a} \frac{z^k}{(1-\t^2)z}\,dz\lesssim e^{\t^a y^2/2}\t^{ka}.$$
In the second integral we use that $z\le \t^2+(1-\t^2)z\le (\t^{2-a}+1-\t^2)z$, for $z\ge\t^a$, to
see that the integral is $1+O(\t^{2-a})$ times
$$\int_{\t^a}^1 \frac{z^k}{z}e^{y^2z/2}\,dz\gtrsim  e^{\t^a y^2/2}.$$
Combining these displays, we see that 
$$I_{k}(y)=\int_{\t^a}^1 z^{k-1}e^{y^2z/2}\,dz(1+O(\t^{2-a})+O(\t^{ka})).$$
This remains valid if we enlarge the range of integration to $[0,1]$.
The  change of coordinates $zy^2/2=v$ completes the proof of the equality in the first assertion.

\beginskip
For $y\ge 1$ the inequality in the first assertion follows for $k\ge 1$ from the bound $\int_0^y v^{k-1}e^v\,dv\le y^{k-1}\int_0^y e^v\,dv$
and for $0<k\le 1$ from the bound $\int_0^y v^{k-1}e^v\,dv\le e^{y/2}\int_0^{y/2} v^{k-1}\,dv+(y/2)^{k-1}\int_{y/2}^ye^v\,dv
\lesssim y^{k-1}e^y$, since $\sup_{y>0} (ye^{-y/2})<\infty$. For $y\le 1$, the bound follows
from $\int_0^y v^{k-1}e^v\,dv\le e^y \int_0^y v^{k-1}\,dv\lesssim e^y y^k$.
\endskip

For the second assertion we expand the integral in the first assertion with the
help of the second assertion of Lemma~\ref{LemmaIncompleteGamma}. Note
here that for $k>-1$ the integrals in the latter lemma can be taken over
$(0,y)$ instead of $(1,y)$, since the difference is a constant. 

The inequality in the first assertion is valid for $y\ra\infty$, in view of the second assertion, and from the fact that
$G(y):= (y^2/2)^{-k}\int_0^{y^2/2} v^{k-1}e^v\,dv$ possesses a finite limit as $y\da0$ it follows that 
it is also valid for $y\ra0$. For intermediate $y$ the inequality follows since the continuous function 
$y\mapsto G(y)e^{-y^2/2}/(y^{-2}\wedge 1)$ is bounded on compacta in $(0,\infty)$.

For the proofs of the assertions concerning $I_{1/2}-I_{3/2}$ we write
$$I_{1/2}(y)-I_{3/2}(y)
=\Bigl(\int_0^{\t}+\int_{\t}^1\Bigr) \frac{\sqrt z(1-z)}{\t^2+(1-\t^2)z}e^{y^2z/2}\,dz.$$
Next we follow the same approach as previously.
\end{proof}

\begin{lemma}\label{Lem: Diff}
For any stochastic process $(V_{\t}: \t>0)$ with continuously differentiable sample paths $\t\mapsto V_\t$, 
with derivative written as $\dot V_\t$,
\begin{align*}
\E( V_{\t_2}-V_{\t_1})^2\leq (\t_2-\t_1)^2\sup_{\t\in[\t_1,\t_2]}\E \dot V_{\t}^2.
\end{align*}
\end{lemma}

\begin{proof}
By the Newton-Leibniz formula, the Cauchy-Schwarz inequality, 
Fubini's theorem and the mean integrated value theorem, for $\t_2\geq\t_1$,
\begin{align*} 
\E\big(V_{\t_1}-V_{\t_2}\big)^2&= \E \big(\int_{\t_1}^{\t_2}\dot V_{\t}\,d\t\big)^2\leq \E (\t_2-\t_1)\int_{\t_1}^{\t_2}\dot V_{\t}^2\,d\t\\
&= (\t_2-\t_1)\int_{\t_1}^{\t_2} \E \dot V_{\t}\, d\t\leq (\t_2-\t_1)^2\sup_{\t\in[\t_1,\t_2]}\E \dot V_{\t}^2\,d\t.
\end{align*}
\end{proof}

\begin{lemma}\label{Lem: DominatingDensity}
If $f_1,f_2: [0,\infty)\ra[0,\infty)$ are probability densities such that  $f_2/f_1$ is monotonely increasing,
then, for any monotonely increasing function $h$,
$$\E_{f_1} h(X)\leq \E_{f_2} h(X).$$ 
\end{lemma}

\begin{proof}
Define $g=f_2/f_1$. Since $\int_0^{\infty}f_1(x)dx=\int_0^{\infty}f_1(x)g(x)\,dx$ and $g$ is monotonely increasing,
there exists an $x_0>0$ such that $g(x)\leq 1$ for $x< x_0$ and $g(x)\geq 1$ for $x> x_0$. Therefore
\begin{align*}
0&=h(x_0)\int_{0}^{\infty}f_1(x)\big(g(x)-1\big)\,dx\\
&\leq \int_{0}^{x_0}f_1(x)h(x)\big(g(x)-1\big)\,dx+\int_{x_0}^{\infty}f_1(x) h(x)\big(g(x)-1\big)\,dx.
\end{align*}
By the definition of $g$ the right side is $\E_{f_2} h(X)-\E_{f_1} h(X)$.
\end{proof}

\bibliographystyle{acm}
\bibliography{references}

\begin{thebibliography}{10}

\bibitem{Armagan2013}
{\sc Armagan, A., Dunson, D.~B., and Lee, J.}
\newblock Generalized double {P}areto shrinkage.
\newblock {\em Statistica Sinica 23\/} (2013), 119--143.

\bibitem{Bhadra2015}
{\sc Bhadra, A., Datta, J., Polson, N.~G., and Willard, B.}
\newblock The horseshoe+ estimator of ultra-sparse signals.
\newblock arXiv:1502.00560v2, 2015.

\bibitem{Bhattacharya2014}
{\sc Bhattacharya, A., Pati, D., Pillai, N.~S., and Dunson, D.~B.}
\newblock Dirichlet-{L}aplace priors for optimal shrinkage.
\newblock arXiv:1401.5398, 2014.

\bibitem{Bickel2009}
{\sc Bickel, P.~J., Ritov, Y., and Tsybakov, A.~B.}
\newblock Simultaneous analysis of {L}asso and {D}antzig selector.
\newblock {\em The Annals of Statistics 37}, 4 (2009), 1705--1732.

\bibitem{Caron2008}
{\sc Caron, F., and Doucet, A.}
\newblock Sparse {B}ayesian nonparametric regression.
\newblock In {\em Proceedings of the 25th International Conference on Machine
  Learning\/} (New York, NY, USA, 2008), ICML '08, ACM, pp.~88--95.

\bibitem{Carvalho2009}
{\sc Carvalho, C.~M., Polson, N.~G., and Scott, J.~G.}
\newblock Handling sparsity via the horseshoe.
\newblock {\em Journal of Machine Learning Research, W\&CP 5\/} (2009), 73--80.

\bibitem{Carvalho2010}
{\sc Carvalho, C.~M., Polson, N.~G., and Scott, J.~G.}
\newblock The horseshoe estimator for sparse signals.
\newblock {\em Biometrika 97}, 2 (2010), 465--480.

\bibitem{Castillo2015}
{\sc Castillo, I., Schmidt-Hieber, J., and van~der Vaart, A.}
\newblock {B}ayesian linear regression with sparse priors.
\newblock {\em Ann. Statist. 43}, 5 (10 2015), 1986--2018.

\bibitem{Castillo2012}
{\sc Castillo, I., and Van~der Vaart, A.~W.}
\newblock Needles and straw in a haystack: Posterior concentration for possibly
  sparse sequences.
\newblock {\em Ann. Statist. 40}, 4 (2012), 2069--2101.

\bibitem{Datta2013}
{\sc Datta, J., and Ghosh, J.~K.}
\newblock Asymptotic properties of {B}ayes risk for the horseshoe prior.
\newblock {\em Bayesian Analysis 8}, 1 (2013), 111--132.

\bibitem{Donoho1992}
{\sc Donoho, D.~L., Johnstone, I.~M., Hoch, J.~C., and Stern, A.~S.}
\newblock Maximum entropy and the nearly black object (with discussion).
\newblock {\em Journal of the Royal Statistical Society. Series B
  (Methodological) 54}, 1 (1992), 41--81.

\bibitem{Ghosal2000}
{\sc Ghosal, S., Ghosh, J.~K., and Van~der Vaart, A.~W.}
\newblock Convergence rates of posterior distributions.
\newblock {\em The Annals of Statistics 28}, 2 (2000), 500--531.

\bibitem{GhosalLemberandvanderVaart(2008)}
{\sc Ghosal, S., Lember, J., and van~der Vaart, A.}
\newblock Nonparametric {B}ayesian model selection and averaging.
\newblock {\em Electron. J. Stat. 2\/} (2008), 63--89.

\bibitem{Ghosh2015}
{\sc Ghosh, P., and Chakrabarti, A.}
\newblock Posterior concentration properties of a general class of shrinkage
  estimators around nearly black vectors.
\newblock arXiv:1412.8161v2, 2015.

\bibitem{Gramacy2014}
{\sc Gramacy, R.~B.}
\newblock {\em monomvn: Estimation for multivariate normal and Student-t data
  with monotone missingness}, 2014.
\newblock R package version 1.9-5.

\bibitem{Griffin2010}
{\sc Griffin, J.~E., and Brown, P.~J.}
\newblock Inference with normal-gamma prior distributions in regression
  problems.
\newblock {\em Bayesian Analysis 5}, 1 (2010), 171--188.

\bibitem{fasthorseshoe}
{\sc Hahn, R.~P., He, J., and Lopes, H.}
\newblock {\em fastHorseshoe: The Elliptical Slice Sampler for Bayesian
  Horseshoe Regression}, 2016.
\newblock R package version 0.1.0.

\bibitem{Jiang2009}
{\sc Jiang, W., and Zhang, C.-H.}
\newblock General maximum likelihood empirical {B}ayes estimation of normal
  means.
\newblock {\em Ann. Statist. 37}, 4 (08 2009), 1647--1684.

\bibitem{Johnson2010}
{\sc Johnson, V.~E., and Rossell, D.}
\newblock On the use of non-local prior densities in {B}ayesian hypothesis
  tests.
\newblock {\em J. R. Stat. Soc. Ser. B Stat. Methodol. 72}, 2 (2010), 143--170.

\bibitem{Johnstone2004}
{\sc Johnstone, I.~M., and Silverman, B.~W.}
\newblock Needles and straw in haystacks: Empirical {B}ayes estimates of
  possibly sparse sequences.
\newblock {\em Ann. Statist. 32}, 4 (2004), 1594--1649.

\bibitem{Makalic2015}
{\sc Makalic, E., and Schmidt, D.~F.}
\newblock A simple sampler for the horseshoe estimator.
\newblock arXiv:1508.03884, 2015.

\bibitem{Polson2010}
{\sc Polson, N.~G., and Scott, J.~G.}
\newblock Shrink globally, act locally: Sparse {B}ayesian regularization and
  prediction.
\newblock In {\em Bayesian Statistics 9}, J.~Bernardo, M.~Bayarri, J.~Berger,
  A.~Dawid, D.~Heckerman, A.~Smith, and M.~West, Eds. Oxford University Press,
  2010.

\bibitem{Polson2012}
{\sc Polson, N.~G., and Scott, J.~G.}
\newblock Good, great or lucky? {S}creening for firms with sustained superior
  performance using heavy-tailed priors.
\newblock {\em The Annals of Applied Statistics 6}, 1 (2012), 161--185.

\bibitem{Polson2012-2}
{\sc Polson, N.~G., and Scott, J.~G.}
\newblock On the half-{C}auchy prior for a global scale parameter.
\newblock {\em Bayesian Analysis 7}, 4 (2012), 887--902.

\bibitem{Rockova2015}
{\sc Ro\u{c}kov\'a, V.}
\newblock Bayesian estimation of sparse signals with a continuous
  spike-and-slab prior.
\newblock submitted manuscript, available at
  \url{http://stat.wharton.upenn.edu/~vrockova/rockova2015.pdf}, 2015.

\bibitem{rousseau:szabo:2015}
{\sc {Rousseau}, J., and {Szabo}, B.}
\newblock {Asymptotic behaviour of the empirical Bayes posteriors associated to
  maximum marginal likelihood estimator}.
\newblock {\em ArXiv e-prints\/} (Apr. 2015).

\bibitem{Scott2010-2}
{\sc Scott, J.~G.}
\newblock Parameter expansion in local-shrinkage models.
\newblock arXiv:1010.5265, 2010.

\bibitem{Scott2011}
{\sc Scott, J.~G.}
\newblock Bayesian estimation of intensity surfaces on the sphere via needlet
  shrinkage and selection.
\newblock {\em Bayesian Analysis 6}, 2 (2011), 307--328.

\bibitem{SzVZ}
{\sc Szabo, B.~T., van~der Vaart, A.~W., and van Zanten, J.}
\newblock Empirical {B}ayes scaling of {G}aussian priors in the white noise
  model.
\newblock {\em Electron. J. Statist. 7\/} (2013), 991--1018.

\bibitem{Tibshirani1996}
{\sc Tibshirani, R.}
\newblock Regression shrinkage and selection via the {L}asso.
\newblock {\em J. R. Stat. Soc. Ser. B Stat. Methodol. 58}, 1 (1996), 267--288.

\bibitem{horseshoepackage}
{\sc {van der Pas}, S., Scott, J., Chakraborty, A., and Bhattacharya, A.}
\newblock {\em horseshoe: Implementation of the Horseshoe Prior}, 2016.
\newblock R package version 0.1.0.

\bibitem{coveragepaper}
{\sc van~der Pas, S., Szab\'o, B., and van~der Vaart, A.}
\newblock Uncertainty quantification for the horseshoe.
\newblock preprint, 2017.

\bibitem{vdPas}
{\sc van~der Pas, S.~L., Kleijn, B. J.~K., and van~der Vaart, A.~W.}
\newblock The horseshoe estimator: Posterior concentration around nearly black
  vectors.
\newblock {\em Electron. J. Statist. 8}, 2 (2014), 2585--2618.

\bibitem{vdVvZGamma}
{\sc van~der Vaart, A., and van Zanten, H.}
\newblock Adaptive {B}ayesian estimation using a {G}aussian random field with
  inverse gamma bandwidth.
\newblock {\em Ann. Statist. 37}, 5B (2009), 2655--2675.

\bibitem{vdVW}
{\sc van~der Vaart, A.~W., and Wellner, J.~A.}
\newblock {\em Weak convergence and empirical processes}.
\newblock Springer Series in Statistics. Springer-Verlag, New York, 1996.
\newblock With applications to statistics.

\end{thebibliography}
\end{document}